\documentclass[fleqn,reqno,11pt]{amsart}
\usepackage{mathtools}
\usepackage{microtype}
\usepackage[pdftex,marginratio=1:1,width=130mm,height=190mm,top=40mm,headsep=17mm,head=7mm,footskip=11.29mm]{geometry}
\usepackage[colorlinks,linkcolor=blue,citecolor=magenta,anchorcolor=blue,bookmarksopen,pdfpagetransition={Wipe}]{hyperref}
\usepackage{graphicx}
\usepackage{algorithm, algorithmic}
\usepackage{numprint}
\usepackage{lscape}
\usepackage{tikz,pgfplots}
\usepackage{color}
\newtheorem{thm}{Theorem}[section]

\newtheorem{lem}[thm]{Lemma}

\theoremstyle{definition}
\newtheorem{defn}[thm]{Definition}

\numberwithin{equation}{section}
\DeclarePairedDelimiter{\norm}{\lVert}{\rVert}

\DeclarePairedDelimiter{\interval}{\langle}{\rangle}

\newcommand{\va}{\operatorname{vec}}
\newcommand{\mi}{\operatorname{mid}}
\newcommand{\ra}{\operatorname{rad}}
\newcommand{\di}{\operatorname{diag}}
\newcommand{\relerr}{\operatorname{relerr}}
\newcommand{\rp}{\operatorname{rp}}
\newcommand{\garp}{\ensuremath{\texttt{garp}}}
\newcommand{\nre}{\ensuremath{\texttt{nre}}}
\newcommand{\A}{\mathbf{A}}
\newcommand{\B}{\mathbf{B}}
\newcommand{\CC}{\mathbf{C}}
\newcommand{\Z}{\mathbf{Z}}

\newcommand{\R}{\mathbb{R}}
\newcommand{\x}{\mathbf{x}}
\newcommand{\X}{\mathbf{X}}
\newcommand{\Y}{\mathbf{Y}}
\title[Verified CARE solutions]{Methods for verified stabilizing solutions to continuous-time algebraic Riccati equations}
\author{Tayyebe Haqiri}
\address{Department of Applied Mathematics, Faculty of Mathematics and Computer\\
Shahid Bahonar University of Kerman, Kerman, Iran\\
\newline Member of Young Researchers Society of Shahid Bahonar University of Kerman\\
Kerman, Iran,\\
\texttt{Haqiri@math.uk.ac.ir, thaqiri@gmail.com}}
\author{Federico Poloni}
\address{Dipartimento di Informatica, Universit\`a di Pisa,\\
Largo B. Pontecorvo 3, 56127 Pisa, Italy\\
\texttt{fpoloni@di.unipi.it}}

\begin{document}
\begin{abstract}
We describe a procedure based on the Krawczyk method to compute a verified enclosure for the stabilizing solution of a continuous-time algebraic Riccati equation $A^*X+XA+Q=XGX$ building on the work of [B.~Hashemi, \emph{SCAN} 2012] and adding several modifications to the Krawczyk procedure. We show that after these improvements the Krawczyk method reaches results comparable with the current state-of-the-art algorithm [Miyajima, \emph{Jpn. J. Ind. Appl. Math} 2015], and surpasses it in some examples.
Moreover, we introduce a new direct method for verification which has a cubic complexity in term of the dimension of $X$, employing a fixed-point formulation of the equation inspired by the ADI procedure. The resulting methods are tested on a number of standard benchmark examples.
\end{abstract}

\maketitle

\maketitle


\section{Introduction}\label{introduction}
Consider the continuous-time algebraic Riccati equation \emph{CARE} 
\begin{equation}\label{care}
A^*X+XA+Q = XGX, 
\end{equation}
where $ A, G$ and $Q \in \mathbb{C}^{n \times n}$ are given, $G$ and $Q$ are Hermitian, and $X \in \mathbb{C}^{n \times n}$ is unknown. Here, the notation $A^*$ denotes the conjugate transpose of a complex matrix $A$ while $A^T$ shows the transpose of $A$. CAREs have a variety of applications in the field of control theory and filter design, such as the linear-quadratic optimal control problem and Hamiltonian systems of differential equations~\cite{LancasterRodman1995}. A solution $X_s$ of~\eqref{care} is called \emph{stabilizing} if the \emph{closed loop matrix} $A-GX_s$ is \emph{Hurwitz stable}, i.e., if all its eigenvalues have strictly negative real part. If a stabilizing solution $X_s$ exists, it is unique~\cite[Theorem~2.17]{BiniIannazzoMeini2012}, and it is Hermitian, i.e., $X_s=(X_s)^*$. The unique stabilizing solution is the one of interest in almost all applications~\cite{LancasterRodman1995,BiniIannazzoMeini2012}, hence in this paper we focus on its computation. The techniques presented here can be adapted with minor sign changes to \emph{anti-stabilizing} solutions, i.e., solutions $X_{as}$ for which all the eigenvalues of $A-GX_{as}$ have positive real part. The algorithms in~\cite{Hashemi2012} and~\cite{Miyajima2015}, in contrast, do not restrict to verifying stabilizing solutions only; however, solutions which are neither stabilizing nor anti-stabilizing have very few applicative uses.

The solutions of~\eqref{care} can be put in one-to-one correspondence with certain invariant
subspaces of the \emph{Hamiltonian}
\[ H := \begin{bmatrix}
	A & -G\\ -Q & -A^*
\end{bmatrix} \in \mathbb{C}^{2n \times 2n}. \]
 Indeed, $X$ is a solution of~\eqref{care} if and only if
\begin{equation} \label{invsubformulation}
H\begin{bmatrix}
I_n \\ X
\end{bmatrix}
=\begin{bmatrix}
I_n \\ X
\end{bmatrix}
(A-GX),
\end{equation}
 in which $I_n$ is the identity matrix in $\mathbb{C}^{n \times n}$. In particular, the columns of the matrix $\begin{bsmallmatrix}
I_n \\ X
\end{bsmallmatrix}$ 
span an invariant subspace for the matrix $H$, and the eigenvalues of $A-GX$ are a subset of the eigenvalues of $H$~\cite{BiniIannazzoMeini2012}. We refer the reader to the books by Lancaster and Rodman~\cite{LancasterRodman1995} and by Bini, Iannazzo and Meini~\cite{BiniIannazzoMeini2012} for details concerning main theoretical properties and numerical solutions together with the description of the main tools for the design and analysis of solution algorithms.

The work presented in this paper addresses the problem of computing verified stabilizing solutions of CAREs~\eqref{care}, that is, determining an interval matrix which is guaranteed to contain the stabilizing solution of the CARE. The problem of computing verified solutions to the 
matrix Riccati equations (AREs) has been addressed before in the literature: the algorithms in~\cite{LutherOttenTraczinski1998} and~\cite{LutherOtten1999}, based on the interval Newton method, are pioneering works in this context but their computational complexity is $\mathcal{O}(n^{6})$. In~\cite{LutherOttenTraczinski1998}, the authors apply Brouwer's fixed point theorem to calculate verified solutions of the ARE 
\begin{equation} \label{ARE} 
A^TX+XA+Q=XBR^{-1}B^TX, 
\end{equation}
with real symmetric matrices $Q$ and $R$, $Q$ positive semi definite and $R$ positive definite. They find an interval matrix including a positive definite solution of~\eqref{ARE}. The paper~\cite{YanoKoga2011} decreases this cost to $\mathcal{O}(n^{5})$ by using the Krawczyk method, which is a variant of the Newton method that does not require the inversion of an interval matrix. A major improvement is the algorithm in~\cite{Hashemi2012}, which is applicable when the closed-loop matrix $A-G\check{X}$ is diagonalizable where $\check{X}$ denotes a numerical computed solution of \eqref{care}, and requires only $\mathcal{O}(n^{3})$ operations. The recent paper~\cite{Miyajima2015} describes a more efficient algorithm based again on the diagonalization of $A-G\check{X},\check{X}$ a Hermitian numerical solution of~\eqref{care}. The resulting method has cubic complexity as well. An important feature of this algorithm is that does not require iteration to find a suitable candidate interval solution, unlike the previous methods based on the Krawczyk method and fixed-point theorems. Hence it is typically faster than the alternatives. The same paper~\cite{Miyajima2015} also includes a method to verify the uniqueness and the stabilizing property of the computed solution.

We propose here a variant of the Krawczyk method suggested in~\cite{Hashemi2012}, introducing several modifications. Namely:
\begin{itemize}
	\item We use the technique introduced in~\cite{Frommer2009}, which consists in applying the Krawczyk method not to the original equation, but to one obtained after a change of basis, in order to reduce the number of verified operations required, with the aim to reduce the wrapping effects.
	\item We exploit the invariant subspace formulation~\eqref{invsubformulation} to make another change of basis in the matrix $H$, following a technique introduced in~\cite{MehrmannPoloni2012} for the \emph{non-verified} solution of Riccati equations. This technique employs suitable permutations of $H$ to reduce~\eqref{care} to a different CARE whose stabilizing solution $Y_s$ has bounded norm. Up to our knowledge, this is the first attempt to use this technique in the context of interval arithmetic and verified computation.
	\item When applying the Krawczyk method, an enclosure for the so-called \emph{slope} matrix is needed; the standard choice to compute it is using the interval evaluation of the Jacobian of the function at hand. Instead, we use a different algebraic expression which results in a smaller interval.
\end{itemize}
An algorithm to verify the stabilizing property (and thus the uniqueness) of the solution enclosure in the computed interval matrix is also presented.

In addition, we present a different algorithm, based on a reformulation of~\eqref{care} as a fixed-point equation, which requires $\mathcal{O}(n^3)$ operations per step and does not require the diagonalizability of the closed-loop matrix $A-G\check{X}$ in which $\check{X}$ is the computed approximate stabilizing solution of CARE~\eqref{care}. This algorithm is generally less reliable than the Krawczyk-based ones, but it has the advantage of not breaking down in cases in which the closed-loop matrix is defective or almost defective.

We conclude the paper by evaluating the proposed algorithms on a large set of standard benchmark problems~\cite{carex1995,ChuLiuMehrmann2007} for Riccati equations, comparing them with the algorithms in~\cite{Hashemi2012} and~\cite{Miyajima2015}. Using all the improvements described here, the gap between the Krawczyk method and the current best method in~\cite{Miyajima2015} is essentially eliminated. The four methods each handle satisfactorily a slightly different set of problems, and none of them is beaten by the alternatives in all possible experiments.

The paper is organized as follows. In the next section we introduce some notation and standard results in linear algebra and interval analysis which are at the basis of our methods. In section~\ref{findingenclosures} we discuss various algorithms based on the Krawczyk method to compute a thin interval matrix enclosing a solution of~\eqref{care} while in section~\ref{directfixedpoint}, a fixed point approach is presented. In section~\ref{numerical} and~\ref{summary} we perform some numerical tests and draw the conclusions and outlook, respectively.


\section{Preliminaries and Notation}\label{notation}
We try to follow the standard notation of interval analysis defined in~\cite{Kearfott2005}. Subsequently, we use boldface lower and upper case letters for interval scalars or vectors and matrices, respectively, whereas lower case stands for scalar quantities and point vectors and upper case represents matrices. By $\mathbb{C}^{m\times n}(\R^{m\times n})$ and $\mathbb{IC}_{\text {disc}}^{m\times n} (\mathbb{I}\R^{m\times n})$ we denote the sets of all complex (real) $m\times n$ matrices and the set of all complex (real) $m\times n$ interval matrices, respectively. By a complex (real) interval matrix, we mean a matrix whose entries are circular complex (compact real) intervals.

The \emph {Kronecker product} $A\otimes B$ of an $m\times n$ matrix $A = (A_{ij})$ and a $p\times q$ matrix $B$ is an $mp\times nq$ matrix defined as the block matrix whose blocks are $A\otimes B:=[A_{ij}B]$. For a point matrix $A \in \mathbb{C}^{m\times n}$, the vector $\va (A) \in \mathbb{C}^{mn} $ denotes column-wise vectorization whereby the successive columns of $A$ are stacked one below the other, beginning with the first column and ending with the last. Moreover, $\overline{A}$ denotes the complex conjugate of $A$ and if $A$ is an invertible matrix, then $A^{-T}:=(A^{T})^{-1}$ and $A^{-*}:=(A^{*})^{-1}$. The element-wise division of a matrix $A = (A_{ij})\in \mathbb{C}^{m\times n}$ by a matrix $B = (B_{ij})\in\mathbb{C}^{m\times n}$, also known as the \emph {Hadamard} division, denoted by $A./ B$, results in an $m\times n$ matrix $C= (C_{ij})$ whose $(i,j)$-th element is given by $C_{ij}=A_{ij}/B_{ij}$ provided that $B_{ij} \neq 0,$ for each $1 \leq i \leq m$ and $1 \leq j \leq n$. For a given vector $d=(d_1,d_2,\dotsc,d_n)^T\in \mathbb{C}^{n}$, $\di(d)\in \mathbb{C}^{n\times n}$ is the diagonal matrix whose $i,i$ entry is $d_{i}$. Conversely, given a diagonal matrix $D$, $\di(D)$ is the vector whose elements are the diagonal entries of $D$. Most of these notions and operations are analogously defined for interval quantities.

Complex intervals can be defined either as rectangles or as discs. We use here the definition as discs: a circular complex interval $\x$, or circular disc or simply a complex interval, is a closed circular disc of radius $\ra {(\x)} \in \mathbb{R}$ with $\ra{(\x)} \geq 0$ and center $\mi{(\x)} \in \mathbb{C}$. Indeed, it is defined as $\x:= \{z \in \mathbb {C}\colon|z-\mi{(\x)}| \leq \ra{(\x)}\}= \left\langle \mi{(\x)}, \ra{(\x)}\right\rangle$.
The operations on the circular complex intervals, $\mathbb{IC}_\text{disc}$, are introduced as generalizations of operations on complex numbers. Then, the standard arithmetic for circular complex interval arguments $\x=\left\langle \mi{(\x)}, \ra{(\x)}\right\rangle$ and $\mathbf{y}=\left\langle \mi{(\mathbf{y})}, \ra{(\mathbf{y})}\right\rangle$ is defined as~\cite{Alefeld1983}:
\[\x \pm \mathbf{y}:=\left\langle\mi{(\x)}\pm \mi{(\mathbf{y})},\ra{(\x)}+\ra{(\mathbf{y})}\right\rangle,\]
\[\x \cdot \mathbf{y}:=\left\langle\mi{(\x)}\mi{(\mathbf{y})},|\mi{(\x)}|\ra{(\mathbf{y})}+|\mi{(\mathbf{y})}|\ra{(\x)}+\ra{(\x)}\ra{(\mathbf{y})}\right\rangle,\]
\[{1}/{\x}:=\left\langle {\overline{\mi{(\x)}}}/{|\mi{(\x)}|^2-(\ra{(\x))^2}},{\ra (\x)}/{|\mi{(\x)}|^2-(\ra{(\x))^2}}\right\rangle,\quad 0 \notin \x,\]
\[{\x}/{\mathbf{y}}:=\x \cdot {1}/{\mathbf{y}},\quad 0\notin\mathbf{y}.\]
Here $|z|$ denotes the absolute value of a complex number $z$. Additionally, operations between a complex interval and a complex number $z \in \mathbb{C}$ can be performed by identifying $z$ with $\interval{z,0} \in \mathbb{IC}_{\text{disc}}$ which can be then extended to other point and interval quantities. Besides, $\overline{z}$ is the complex conjugate of $z$.\\
We emphasize that the definitions of addition, subtraction and inversion substantially coincide with their set theoretic definitions, i.e.,
\[\x \pm \mathbf{y} = \{x \pm y \colon x \in \x, y \in \mathbf{y}\},\]
\[ {1}/{\x} = \{{1}/{x} \colon x \in \x \}.\]
Unfortunately, the set $\{xy \colon x \in \x, y \in \mathbf{y}\}$, in general, is not a disc but one of the basic properties of interval arithmetic, which makes its use well-founded, is that respects inclusion: for \emph{all} the four basic arithmetic operations $\circ \in \{+,-,\cdot,/\}$ one has 
\[\x \circ \mathbf{y} \supseteq \{x \circ y~:~ x \in \x,~y \in \mathbf{y}\},\]
in which $\x$ and $\mathbf{y}$  are two real or circular complex intervals. (In the case of division, we need to assume that $0 \notin \mathbf{y}$ for the operation to be well-defined.) Hence we have the following \emph{inclusion property}: if $\mathbf{f}$ is the \emph{interval evaluation} of $f$ (an ordinary real or complex  function of $N$ variables), then
\begin{multline*}
 f(\x_1,\x_2,\dotsc,\x_N) := \{f (x_1,x_2,\dotsc,x_N): x_1\in\x_1, x_2\in\x_2,\dots, x_N\in\x_N\}\\ \subseteq \mathbf{f}(\x_1,\x_2,\dotsc,\x_N).
\end{multline*}
The interval evaluation~\cite[e.g.]{Moore2009} of a function $f(x_1,x_2,\dots,x_N)$, $\mathbf{f}(\x_1,\x_2,\dotsc,\x_N)$, defined by a formula is obtained by replacing (1) the variables $x_1,x_2,\dots,x_N$ with interval variables $\x_1,\x_2,\dots,\x_N$ and (2) each arithmetic operation with the corresponding interval operation. We will utilize the same approach to define the interval evaluation of a matrix function as well. Note that, in principle, different equivalent formulas could give different interval evaluations; indeed, the process of turning the customary arithmetic into interval arithmetic is not free of pitfalls; issues such as \emph{interval dependency} and the \emph{wrapping phenomenon} have to be considered carefully. We refer the reader to the review article~\cite{Rump2010} for a thorough introduction.

The intersection of two complex intervals $\x$ and $\mathbf{y}$ is not always a complex interval so we may define it to be any complex interval $\mathbf{z}$ such that $\x \cap \mathbf{y} \subseteq \mathbf{z}$, if they are not disjoint. If they are disjoint, the intersection is the empty set. The function $\texttt{intersect.m}$ in INTLAB can be used to obtain a tight intersect. The interval hull of two intervals $\x$ and $\mathbf{y}$ is denoted by $\square {(\x,\mathbf{y})}$ which is the smallest interval containing $\x$ and $\mathbf{y}.$ The magnitude of $\x \in \mathbb{IC}_\text{disc}$ is defined as $\operatorname{mag}(\x):=\operatorname{max} \{|x|:x \in \x\}$.

We denote by $\A = \interval{\mi(\A),\ra(\A)} \in \mathbb{IC}^{m\times n}_{\text{disc}}$ the $m \times n$ interval matrix $\A$ whose $(i,j)$ element is the complex interval $\interval{\mi(\A_{ij}),\ra(\A_{ij})}$, with $\ra(\A_{ij}) \geq 0$; $1 \leq i \leq m, 1 \leq j \leq n$. For interval vectors and matrices, $\mi, \ra, \operatorname{mag},$ intersect and $\square$ will be applied component-wise.

The \emph{Frobenius norm} of a complex matrix $A=(A_{ij})$ is defined as $\norm{A}_{F} := (\sum_{i,j} A_{ij}^2)^{1/2}$. This definition can be extended to complex interval matrices, providing an interval-valued function $\norm{\A}_{F}$ defined as the smallest interval containing $\{x : x  = \norm{A}_{F}, A \in \A \}$.

The definition of inverse of an interval matrix may be problematic in general, but if $\mathbf{D}= \di ( \mathbf{d})$ is a diagonal interval matrix, with $\mathbf{d} = (\mathbf{d}_1,\mathbf{d}_2,\dots,\mathbf{d}_N)^T$ and $0\not\in \mathbf{d}_i$ for each $i=1,2,\dots,N$, then we may define $\mathbf{D}^{-1} := \di ((\mathbf{d}_1^{-1}, \mathbf{d}_2^{-1},\dots,\mathbf{d}_N^{-1})^T)$.\\

The following lemmas contain simple arithmetical properties of the Kronecker products which we will use in the following. Most of them appear also e.g. in~\cite{Frommer2009} or~\cite{Horn1994}.
\begin{lem} \label{kroneckervec}
Assume that $A=(A_{ij})$, $B=(B_{ij})$, $C=(C_{ij})$ and $D=(D_{ij})$ be complex matrices with compatible sizes. Then,
\begin{enumerate}
	\item $(A\otimes B)(C\otimes D)=AC\otimes BD$,
	\item $A\otimes (B+C)=(A\otimes B)+(A\otimes C)$,
	\item $(A\otimes B)^*=A^*\otimes B^*$,
	\item $(A\otimes B)^{-1}=A^{-1}\otimes B^{-1}$, if $A$ and $B$ are invertible,
	\item $\va(ABC)=(C^T\otimes A) \va {(B)}$,\label{five}
	\item $(\di(\va {(A)}))^{-1} \va {(B)} = \va {(B ./ A)}$, if $A_{ij}\neq 0$ for each $i,j$.\label{vecHadamard}
\end{enumerate}
\end{lem}
\begin{lem} \label{kroneckervecint}
Let $\A=(\A_{ij})$, $\B=(\B_{ij})$ and $\CC=(\CC_{ij})$ be complex interval matrices of compatible sizes. Then,
\begin{enumerate}
	\item $\Bigl\{(C^T\otimes A) \va {(B)} \colon A \in \A,~B \in \B,~C \in \CC\Bigr\} \subseteq \begin{cases}
\va\Bigl(\A(\B\CC)\Bigr),\\
\va\Bigl((\A\B)\CC\Bigr),
\end{cases}$\label{vecmultiproduct}
	\item $\Bigl(\di(\va {(\A)})\Bigr)^{-1} \va {(\B)}= \va ({\B ./ \A})$, if $0\notin \A_{ij}$ for all $i,j$.\label{vecHadamarhint}
\end{enumerate}
\end{lem}

The CPU rounding mode called \emph{round to nearest} is typically used when we compute using a computer. Two other rounding modes are \emph{round downward} (round towards the largest floating point number smaller than the true result), and \emph{round upward} (round towards the smallest floating point number larger than the true result)~\cite{IEEE1985}. In general, it is impossible to compute the exact solution in a numerical computation using floating point numbers. Then, a numerical computation with guaranteed accuracy switches the CPU rounding mode, computes a lower and an upper bound to the true solution, and creates an interval which is guaranteed to contain it~\cite{Hargreaves2002}.
One example of software which provides a fast implementation of such a reliable interval arithmetic is the MATLAB toolbox INTLAB~\cite{Intlab1999}; older versions of INTLAB are freely available for noncommercial use. The default arithmetic for both real and complex intervals in INTLAB is the midpoint-radius arithmetic~\cite{Rump2010}.


\section{Modified Krawczyk's methods} \label{findingenclosures}%

Enclosure methods using interval arithmetic work this way: let $g:\mathbb {C}^N\to \mathbb {C}^N$ be some function of which we wish to find a zero. First find a function $h:\mathbb {C}^N\to \mathbb {C}^N$ whose fixed points are known to be the zeros of $g$. Assume that $h$ is continuous and that we know an enclosure function $ \mathbf{h}$ for $h$, i.e., a function based on `correct' interval arithmetic which gives an interval vector $\mathbf{h} (\x)$ containing the range of $h$ over a given interval vector $\x$. Then if $\mathbf{h} (\x)\subseteq \x$ we know that $h(\x)\subseteq \x$ and so $h$ has a fixed point in $\x$ by Brouwer's theorem~\cite{Frommer2001}. 

In this paper, often the functions $\mathbf{h}$ are variants of the Krawczyk operator. To define this operator, we first need the concept of a slope.
\begin{defn}\label{slope}~\cite[e.g.]{Moore2009}
Suppose $f:\psi \subseteq \mathbb{C}^{N}\to \mathbb{C}^{N}$ and $x, y\in  \mathbb{C}^N$. Then, a \emph{slope} $S(f;x,y)$ is a mapping from the Cartesian product $\psi \times \psi$ to $\mathbb{C}^{N \times N}$ such that
\[f(y)- f(x)=S(f;x,y)(y-x).\]
\end{defn}

We are now ready to state the result which is at the basis of all the modified Krawczyk-type algorithms used in the rest of our paper.
\begin{thm}\label{modifiedkrawczyk}
~\cite{Frommer2014} Assume that $f:\psi \subset \mathbb{C}^N\to \mathbb{C}^N$ is continuous. Let ${\check{x}}\in \psi$ and $\mathbf{z}\in \mathbb{IC}_{\text {disc}}^N$ be such that $\check{x}+\mathbf{z}\subset \psi$. Moreover, assume that $\mathcal{S}\subset \mathbb{C}^{N\times N}$ is a set of matrices such that $S(f;\check{x},x') \in \mathcal{S}$ for every $x'\in \check{x}+\mathbf{z}=:\x$. Finally, let $R\in \mathbb{C}^{N\times N}$. Denote by $\mathcal{K}_f(\check{x},R,\mathbf{z},\mathcal{S})$ the set
\[\mathcal{K}_f(\check{x},R,\mathbf{z},\mathcal{S}):=\{-Rf(\check{x})+(I_N-RS)z~:~S\in \mathcal{S},~z\in \mathbf{z}\}.\]
If
\begin{equation}\label{crucialmodifiedKrawczyk}
\mathcal{K}_f (\check{x},R,\mathbf{z},\mathcal{S})\subseteq \operatorname{int} (\mathbf{z}),   
\end{equation}
then the function $f$ has a zero $x_*$ in $\check{x}+\mathcal{K}_f(\check{x},R,\mathbf{z},\mathcal{S})\subseteq \x$, in which $\operatorname{int} (\mathbf{z})$ is the topological interior of $\mathbf{z}$.\\
Moreover, if $S(f;y,y') \in \mathcal{S}$ for each $y, y'\in \x$, then $x_*$ is the only zero of $f$ contained in $\x$.
\end{thm}

In computation, one defines the Krawczyk operator~\cite{Krawczyk1969}
\begin{equation}\label{standardKrawczyk}
\mathbf{k}_f(\check{x},R,\mathbf{z},\mathbf{S}):=-Rf(\check{x})+(I_N-R\mathbf{S})\mathbf{z},\end{equation}
where $\mathbf{S}$ is an interval matrix containing all slopes $S(f;y,y')$ for $y,y'\in\x$. A common choice for $\mathbf{S}$ is obtained from $\mathbf{f}'(\x)$, an interval evaluation of the Jacobian $f'$ on the interval $\x$. By the inclusion property of interval arithmetic,
\begin{equation}\label{computationalexistence}
\mathbf{k}_f(\check{x},R,\mathbf{z},\mathbf{S})\subset \operatorname{int}\, (\mathbf{z})
\end{equation}
implies~\eqref{crucialmodifiedKrawczyk}. So, if~\eqref{computationalexistence} is satisfied then $f$ has a zero in $\check{x}+\mathbf{k}_f (\check{x},R,\mathbf{z},\mathbf{S})$. In practice, one attempts to make the terms $-Rf(\check{x})$ and $I_N-R\mathbf{S}$ as small as possible, to obtain the crucial relation~\eqref{computationalexistence}. The typical choice is taking as $\check{x}$ a good approximation of a zero of $f$ and as $R$ a good approximation of $(f'(\check{x}))^{-1}$, both obtained via a classic floating point algorithm, see for instance~\cite{Frommer2009}.


\subsection{A Residual form for the Krawczyk operator}\label{residualKrawczykapproach}
We now introduce the concepts that are needed to apply the modified Krawczyk method to solve a matrix equation such as~\eqref{care}. The \emph {Fr\'echet derivative}~\cite{Higham2008} of a Fr\'echet differentiable matrix function $F:\mathbb{C}^{n\times n}\to \mathbb{C}^{n\times n}$ at a point $X \in \mathbb{C}^{n\times n}$ is a linear mapping $L_F :\mathbb{C}^{n\times n} \to \mathbb{C}^{n\times n}$
such that for all $E \in \mathbb{C}^{n\times n}$   
\[F(X + E)- F(X)- L_F(X,E)= o(\norm{E}).\] 
Since $L_{F}$ is a linear operator, we can write
\[\va(L_{F}(X,E)) = K_F(X) \va(E),\]
for a matrix $ K_F(X) \in \mathbb{C}^{n^2 \times n^2}$ that depends on $L$ but not $E$. One refers to $K_F(X)$ as the \emph{Kronecker
form} of the Fr\'echet derivative of $F$ at $X$.

In the case of the continuous-time algebraic Riccati equation~\eqref{care}, we apply the Krawczyk method to the function $F:\mathbb{C}^{n\times n} \to \mathbb{C}^{n\times n}$ defined as
\[
F(X) := A^*X+XA+Q-XGX,
\]
which appeared before in~\cite{Hashemi2012}. For this function, one has
\[L_F(X,E)=E(A-GX)+(A^*-XG) E.\]
Lemma~\ref{kroneckervec} part~\ref{five} turns out that its Kronecker form is
\begin{equation*}
K_F(X)=I_n\otimes (A^*-XG)+(A-GX)^T\otimes I_n.	
\end{equation*}
When $X=X^*$, we can write this expression in an alternate form as
\begin{equation} \label{kronform}
K_F(X)=I_n\otimes (A-GX)^*+(A-GX)^T\otimes I_n.	
\end{equation}

We wish to use the modified Krawczyk algorithm on the function obtained by regarding $F$ as a vector map $f:\mathbb{C}^{N} \to \mathbb{C}^{N}$, with $N=n^2$, defined by
\begin{equation} \label{deff}
	f(x) := \va(A^*X+XA+Q-XGX), \quad \text{$x = \va (X)$}.
\end{equation}
The following result, which is a slight variation of a theorem in~\cite{Hashemi2012}, shows that the Fr\'echet derivative can be used to obtain an enclosure for the slope in the modified Krawczyk method. We report it, with a different proof from the one in~\cite{Hashemi2012}, because this presentation will be more convenient in the following development of our method. Due to this reformulation, we will get a weaker result with respect to uniqueness.
\begin{thm} \label{scan2012}
Let $\X$ be an interval matrix, and $\mathbf{K}_F(\X) = I_n\otimes (A-G\X)^*+(A-G\X)^T\otimes I_n$ be the interval evaluation of $K_F(X)$ in~\eqref{kronform}. Then for each $Y,Y'\in \X$ such that $Y=Y^*$, it holds that $S(f; y, y') \in \mathbf{K}_F(\X)$, where $y=\va(Y),y'=\va(Y')$.
\end{thm}
\begin{proof}
We have
\begin{align*}
\va(F(Y)-F(Y')) &= \va((A^*-YG)(Y-Y') + (Y-Y')(A-GY')) \\
& = \va((A-GY)^*(Y-Y') + (Y-Y')(A-GY')) \\
& = (I_n\otimes (A-GY)^* + (A-GY')^T\otimes I_n)\va(y-y'),
\end{align*}
hence by the inclusion property of interval arithmetic 
\begin{align} \label{herminclusion}
S(f;y,y') = (I_n\otimes (A-GY)^* + (A-GY')^T\otimes I_n) \in \mathbf{K}_F(\X).
\end{align}
\end{proof}

The next ingredient that we need to apply the Krawczyk algorithm is the matrix $R$. One would like to use $R \approx (K_F(\check{X}))^{-1}$, where $\check{X}=\check{X}^*$ is an approximation of the stabilizing solution to the CARE~\eqref{care} computed in floating point arithmetic. However, this is the inverse of an $n^2\times n^2$ matrix, whose computation would cost $\mathcal{O}(n^6)$ floating point operations in general. Even considering the Kronecker product structure of $K_F({\check{X}})$, there is no algorithm in literature to compute $R$ explicitly with less than $\mathcal{O}(n^5)$ arithmetic operations. The action of $R$, that is, computing the product $Rv$ given a vector $v\in\mathbb{C}^{n^2}$, can be computed with $\mathcal{O}(n^3)$ operations with methods such as the Bartels-Stewart algorithm~\cite{BartelsStewart1972}. However, this method cannot be used effectively in conjunction with interval arithmetic due to excessive wrapping effects, as argued in~\cite{Frommer2012}.

The work~\cite{Hashemi2012} (and, earlier, on a similar equation, \cite{Frommer2009}) contains an alternative method to perform this computation with complexity $\mathcal{O}(n^3)$, in the case when $A-G\check{X}$ is diagonalizable, where $\check{X}$ is a numerical solution of CARE~\eqref{care}. Assume that an approximate eigendecomposition of $A-G\check{X}$ is available, that is,
\begin{subequations} \label{defVW}
\begin{gather}
A-G\check{X}  \approx V\Lambda W \quad \text{with $V,W,\Lambda \in \mathbb{C}^{n\times n}$}, \\
\Lambda \approx \di(\lambda_{1},\lambda_{2},\dotsc,\lambda_{n}), \quad VW\approx I_n.
\end{gather}
\end{subequations}
We write $\approx$ instead of $=$ because $V$, $W\approx V^{-1}$ and $\lambda_i$, $i=1,\dots,n$ are computed numerically with a standard method such as MATLAB's~\texttt{eig}. So, equality does not hold (in general) in the mathematical sense. Once these quantities are computed, we can factorize $K_F(\check{X})$ by replacing $I_n$ in the first term of~\eqref{kronform} with $V^{-T}I_nV^T$ and in the second term with $W^*I_nW^{-*}$ and then using Lemma~\ref{kroneckervec}, so that
\begin{align*}
K_F({\check{X}})&=I_n \otimes (A-G\check{X})^*+(A-G\check{X})^T \otimes I_n\\
&\notag=(V^{-T} \otimes W^*)(I_n \otimes (W(A-G\check{X})W^{-1})^*\\
&\notag +(V^{-1}(A-G\check{X})V)^T \otimes I_n)(V^T \otimes W^{-*}),
\end{align*}
and choose $R$ as
\begin{align} \label{Rmainenclosure}
R=(V^{-T}\otimes W^*)\Delta^{-1}(V^T\otimes W^{-*}),
\end{align}
where $\Delta=I_n\otimes \Lambda^*+\Lambda^T\otimes I_n.$ Then, $R\approx (K_F(\check{X}))^{-1}$ holds since if $\check{X}$ is close enough to the stabilizing solution of~\eqref{care}, then one can expect that $W (A-G\check{X}) W^{-1}$ and also $V^{-1} (A-G\check{X}) V$ to be close to $\Lambda$. So, the computation of an enclosure, $\va (\mathbf{L})$, for $l := -Rf(\check{x})$ in $\mathcal{K}_f(\check{x},R,\mathbf{z},\mathcal{S})$ can be done using exclusively the matrix-matrix operations, as shown in Lines~\ref{computingstabsol}-\ref{firstterm} of Algorithm~\ref{algo:mainenclosure}.

For the latter term in each member of $\mathcal{K}_f(\check{x},R,\mathbf{z},\mathcal{S})$, however, we get
\begin{align*}
u := (I_{n^2}-RS)z &= (I_{n^2}- (V^{-T} \otimes W^*) \Delta^{-1} (V^T \otimes W^{-*})\\
&\notag(I_n \otimes (A-GY)^* + (A-GY^{'})^T \otimes I_n))z\\
&\notag = ((V^{-T} \otimes W^*) \Delta^{-1}\\
&\notag(\Delta - I_n \otimes (W (A-G Y) W^{-1})^* - (V^{-1} (A-GY^{'}) V)^T \otimes I_n)\\
&\notag(V^T \otimes W^{-*}))z, 
\end{align*}
in which $I_{n^2}$ has replaced by $V^{-T}V^T \otimes W^*W^{-*}$, and $Y,Y^{'} \in \X$ with $Y=Y^*$. Then, Algorithm~\ref{algo:krawczykop} Lines~\ref{first}-\ref{Hadamard2} will compute an enclosure for this term as the interval matrix $\mathbf{U}$ whose vectorization contains this term.\\

Another point to note is that we can transform the multiplication $\Gamma^{-1}\va(M)$, for an $n\times n$ matrix $M$ and a diagonal matrix $\Gamma$, into $M ./ N$, where $N$ is defined by $N_{ij}=\bar{\Gamma}_{ii}+\Gamma_{jj}$, using point~\ref{vecHadamard} of Lemma~\ref{kroneckervec}, and similarly for interval matrices using point~\ref{vecHadamarhint} of Lemma~\ref{kroneckervecint}. This point will appear in, for example, Algorithms~\ref{algo:mainenclosure} Line~\ref{Hadamard1}, and Algorithm~\ref{algo:krawczykop} Line~\ref{Hadamard2}.

The standard method~\cite{Rump2010} to obtain an interval vector $\mathbf{z}=\va(\Z)$ that satisfies~\eqref{computationalexistence} is an iterative one. We start from the residual matrix $\Z_0 := \mathbf{F}(\mathbf{\check{X}})$, that is, the interval evaluation of $F(\check{X})$, and proceed alternating successive steps of enlarging this interval with a technique known as \emph{$\varepsilon$-inflation}~\cite{Rump2010}, applying the Krawczyk operator to it, $\mathbf{z}_{i+1} = \mathbf{k}_f (\check{x},R,\mathbf{z}_i,\mathbf{S})$. This procedure terminates when (and if) we find an interval for which~\eqref{computationalexistence} holds; it is ultimately a trial-and-error procedure, which is not guaranteed to succeed: the operator $\mathbf{k}_f$ may simply not contract its interval argument $\mathbf{z}_i$ sufficiently. This may be due to ill-conditioning of the original equation, to a bad choice of $R$, or to the excessive growth of the intervals in the numerical computations (wrapping effect).

Several slightly different versions of the iterative procedure to obtain a valid interval for inclusion appear in literature; some involve intersecting the intervals obtained in different steps~\cite{Frommer2009,Frommer2014,HashemiDehghan2010}, and some involve two attempts at inclusion in each iteration~\cite{Frommer2009,HashemiDehghan2010}. We use here the simplest approach, following~\cite{Hashemi2012,Rump2010}. The exact strategy is shown in Algorithm~\ref{algo:mainenclosure} (and its subroutine Algorithm~\ref{algo:krawczykop}). The algorithm with these choices coincides with the algorithm presented in~\cite{Hashemi2012}, except for the fact that~\cite{Hashemi2012} presents it for a generic Hermitian solution.

 In all algorithms, whenever the evaluation order of an expression is not specified exactly due to missing brackets, we evaluate from left to right.

\begin{algorithm}[h]
\caption{Computation of an interval matrix $\X$ containing a solution of CARE~\eqref{care}.}
\label{algo:mainenclosure}
\begin{algorithmic}[1]
\STATE Compute an approximate stabilizing solution~ $\check{X}$ of CARE \eqref{care} using any floating point algorithm \label{computingstabsol}
\STATE Compute approximations $V$, $W$, $\Lambda$ for the eigendecomposition of
 $A-G\check{X}$ in floating point \COMMENT{For instance, using the MATLAB command \texttt{eig}}
\STATE	Compute $D:=(D_{ij})$ such that $D_{ij} = \bar{\Lambda}_{ii} + \Lambda_{jj}$
\STATE  Compute interval matrices $\mathbf{I}_{V}$ and $\mathbf{I}_{W}$ containing $V^{-1}$ and $W^{-1}$, respectively \COMMENT{For instance, using \texttt{verifylss.m} from INTLAB.} If this fails, or if $D$ has any zero elements, return \textbf{failure}
\STATE  $\mathbf{\check{X}} = \interval{\check{X},0}$ \COMMENT{To ensure that operations involving $\check{X}$ are performed in a verified fashion with interval arithmetic}
\STATE	$\mathbf{F}=A^*\mathbf{\check{X}}+\mathbf{\check{X}}A + Q -\mathbf{\check{X}} G\mathbf{\check{X}}$ \COMMENT{Using verified interval arithmetic}
\STATE	$\mathbf{G}=\mathbf{I}_{W}^* \mathbf{F}V$
\STATE	$\mathbf{H}=\mathbf{G}./D$\label{Hadamard1}
\STATE  $\mathbf{L}=-W^*\mathbf{H} \mathbf{I}_V$ \label{firstterm}
\STATE  $\Z=\mathbf{L}$
\FOR    {$k=1,\dotsc,k_{max}$}
\STATE	Set $\Z=\square (0,\Z\cdot \interval{1,0.1} +\interval{0,\texttt{realmin}})$ \COMMENT{$\varepsilon$-inflation technique}
\STATE	Compute $\mathbf{K}$ using Algorithm~\ref{algo:krawczykop}
\IF {$\mathbf{K} \subset \operatorname{int} (\Z)$ \{successful inclusion\}} 
\STATE Return $\X=\check{X}+\mathbf{K}$
\ENDIF
\STATE  $\Z=\mathbf{K}$
\ENDFOR
\STATE Return \textbf{failure} \COMMENT{Maximum number of iterations reached}
\end{algorithmic}
\end{algorithm}

\begin{algorithm}[h]
\small
\caption{Computation of an interval matrix $\mathbf{K}$ such that $\va {(\mathbf{K})}=\mathbf{k}_{f}(\check{x},R,\mathbf{z},\mathbf{S})$ encloses $\mathcal{K}_{f}(\check{x},R,\mathbf{z},\mathcal{S})$.}
\label{algo:krawczykop}
\begin{algorithmic}[1]
\STATE   Input $A$, $G$, $Q$, $\check{X}$, $\Z$\\
\COMMENT{Additionally, in this subfunction we use $V,W,\mathbf{I}_V,\mathbf{I}_W,\Lambda,D,\mathbf{L}$ which are already computed in Algorithm~\ref{algo:mainenclosure}}
\STATE	$\mathbf{M}=\mathbf{I}_{W}^*\mathbf{Z}V$\label{first}
\STATE	$\mathbf{N}=W(A-G(\check{X}+\Z))\mathbf{I}_W$ \label{supersetline}
\STATE	$\mathbf{O}=\mathbf{I}_{V}(A-G(\check{X}+\Z))V$
\STATE	$\mathbf{P}=(\Lambda-\mathbf{N})^*\mathbf{M}+\mathbf{M}(\Lambda-\mathbf{O})$
\STATE	$\mathbf{Q}=\mathbf{P}./D$\label{Hadamard2}
\STATE  $\mathbf{U}=W^*\mathbf{Q}\mathbf{I}_V$
\STATE  $\mathbf{K}=\mathbf{L}+\mathbf{U}$
\STATE  Return $\mathbf{K}$
\end{algorithmic}
\end{algorithm}

Notice the $\varepsilon$-inflation, which is performed by adding $\interval{0,\texttt{realmin}}$ to the computed interval. Throughout the paper, \texttt{realmin} denotes the smallest positive normalized floating point number. 

All the operations in Algorithm~\ref{algo:mainenclosure} are matrix-matrix computations requiring $\mathcal{O}(n^3)$ arithmetic operations, so its total cost is $\mathcal{O}(n^3s)$, where $s$ is the number of steps needed before success.


\subsection{Affine transform enclosure}\label{affinetransform}
The main difficulty in using interval arithmetic to verify the existence of a solution is the so-called \emph{wrapping effect}~\cite{Rump2010}: given interval matrices $\A$ and $\B$, the set $\{AB \colon A\in\A,B\in\B\}$ is not (in general) an interval matrix, and to represent it in interval form we have to enlarge it by replacing with an enclosing interval. The same effect happens with most interval operations, not only matrix multiplications, and it is more pronounced in presence of repeated successive operations, intervals with large radius, and ill-conditioned matrices. This increase may prevent us to verify computationally the critical condition \eqref{crucialmodifiedKrawczyk}.

Reducing the impact of the wrapping effect can give a reduction in the diameter of the computed solution enclosure and also in the computational time, since it can reduce the number of iterations needed before a successful inclusion is computed.

In this section we describe a technique for reducing the wrapping effect in the modified Krawczyk method, which has already been successfully applied to several matrix equations~\cite{Frommer2009,Frommer2014}. The main idea is applying the verification algorithm to a modified function $\hat{f}$ obtained from $f$ via an affine transformation; in this way, we reduce the number of interval operations to perform inside the verification procedure.

Assuming that $V$ and $W$ defined in~\eqref{defVW} are nonsingular, we define the function 
\begin{equation}\label{carehat}
\hat{f}(\hat{x}):=(V^T\otimes W^{-*} )f((V^{-T}\otimes W^*) \hat{x}).
\end{equation}
If $\check{x}=\va (\check{X})$ is an approximate solution to $f(x)=0$, then $\hat{\check{x}}:=(V^T\otimes W^{-*})\check{x}$ is an approximate solution to $\hat{f}(\hat{x})=0$. The Kronecker form of its matrix formulation $\hat{F}(\hat{X})$ is given by
\[
K_{\hat{F}}(\hat{X}) = (V^T\otimes W^{-*} ) K_F(X) (V^{-T}\otimes W^*), \quad X = W^{*}\hat{X}V^{-1}.
\]
Moreover, let $\hat{\x} = \va(\hat{\X}) := \hat{\check{x}} + \hat{\mathbf{z}}$, where $\hat{\mathbf{z}} = \va(\hat{\Z})$. A set of slopes for $\hat{f}$ on $\hat{\x}$ can be defined as 
\[\hat{\mathcal{S}}:=\{S(\hat{f};\hat{y},\hat{y}') \colon \hat{y},\hat{y}'\in \hat{\x} \}.\]
Defining $y:=(V^{-T}\otimes W^*)\hat{y}$, $y':=(V^{-T}\otimes W^*)\hat{y}'$, we have 
\begin{align*}
S(\hat{f};\hat{y},\hat{y}')(\hat{y}-\hat{y}')&=\hat{f}(\hat{y})-\hat{f}(\hat{y}')\\
&=(V^T\otimes W^{-*})(f(y)-f(y'))\\
&=(V^T\otimes W^{-*})S(f;y,y')(y-y')\\
&=(V^T\otimes W^{-*})S(f;y,y')(V^{-T}\otimes W^*)(\hat{y}-\hat{y}').
\end{align*}
Hence 
\begin{equation*}
S(\hat{f};\hat{y},\hat{y}')=(V^T\otimes W^{-*})S(f;y,y')(V^{-T}\otimes W^*).	
\end{equation*}
In particular, if we combine this result with Theorem~\ref{scan2012}, we can take in the Krawczyk operator~\eqref{standardKrawczyk}
\begin{equation}\label{slopehat}
\hat{\mathbf{S}} := \mathbf{S}(\hat{f}) = I_n\otimes (W(A-G \mathbf{\check{X}}) W^{-1})^*+(V^{-1} (A-G\mathbf{\check{X}})V)^T\otimes I_n.
\end{equation}
where
\[\mathbf{\check{X}} =  W^{*} \mathbf{\hat{\check{X}}} V^{-1}, \quad \mathbf{\hat{\check{X}}} = \hat {\check{X}} + \hat{\Z},\]
as long as $\check{X}$ is Hermitian.

Observe that
\[
I_n\otimes (W(A-G\check{X}) W^{-1})^*+(V^{-1} (A-G\check{X})V)^T\otimes I_n \approx I_n\otimes \Lambda^* + \Lambda^T\otimes I_n,
\]
so a natural choice for $\hat{R}$ is the diagonal matrix
\[
\hat{R} = \Delta^{-1},
\]
in which $\Delta$ is defined as in~\eqref{Rmainenclosure}.  

Now, we compute an enclosure for $\mathcal{K}_{\hat{f}}(\hat{\check{x}},\hat{R},\hat{\mathbf{z}},\hat{\mathcal{S}}):= \{-\hat{R}\hat{f}(\hat{\check{x}})+(I_n-\hat{R}S)\hat{z}, S \in \hat{\mathcal{S}}, \hat{z} \in \hat{\mathbf{z}}\}$ which can be written as $\mathbf{k}_{\hat{f}}(\hat{\check{x}},\hat{R},\hat{\mathbf{z}},\hat{\mathbf{S}})$ in which $\hat {\check{x}}$ is an approximate solution for \eqref{carehat}, $\hat{R}$ is $\Delta^{-1}$, $\hat{\mathcal{S}}=\{S(\hat{f};\hat{y},\hat{y}'),\,\hat{y},\hat{y}'\in \hat{\x}:=(V^{T}\otimes W^{-*}){\check{x}}+\hat{\mathbf{z}}\}$, and $\hat {\mathbf{z}}:=\va {(\hat {\Z})}$. As in Algorithm~\ref{algo:mainenclosure}, we also take care that the quantities which are not available exactly are enclosed into computable quantities in interval forms, for instance $\mathbf{I}_V$ and $\mathbf{I}_W$ are interval matrices which are known to contain the exact value of $V^{-1}$ and $W^{-1}$, appropriately. More details for computing the superset
\begin{align*}
\mathbf{k}_{\hat{f}}(\hat{\check{x}},\hat{R},\hat{\mathbf{z}},\hat{\mathbf{S}})&=-\hat{R}\hat{f}(\hat{\check{x}})+(I_n-\hat{R}\hat{\mathbf{S}})\hat{\mathbf{z}}\\
&=-\Delta^{-1} ((V^T\otimes W^{-*})f(\check{x})\\
&\qquad-(\Delta-I_n\otimes (W({A}-G\mathbf{\check{X}}) W^{-1})^*\\
& \qquad-(V^{-1}(A-G\mathbf{\check{X}}) V)^T\otimes I_n) \hat{\mathbf{z}}),
\end{align*}
for $\mathcal{K}_{\hat{f}}(\hat{\check{x}},\hat{R},\hat{\mathbf{z}},\hat{\mathcal{S}})$, are displayed in Algorithm~\ref{algo:krawczykophat}. The complete algorithm is shown in Algorithm~\ref{algo:mainenclosurehat}.

\begin{algorithm}[h]
\caption{Computation of an interval matrix $\X$ containing a solution of CARE~\eqref{care}.}
\label{algo:mainenclosurehat}
\begin{algorithmic}[1]
\STATE Compute an approximate stabilizing solution~$\check{X}$ of CARE \eqref{care} using any floating point algorithm 
\STATE Compute approximations $V$, $W$, $\Lambda$ for the eigendecomposition of
 $A-G\check{X}$ in floating point \COMMENT{For instance, using the MATLAB command \texttt{eig}}
\STATE	Compute $D:=(D_{ij})$ such that $D_{ij} = \bar{\Lambda}_{ii} + \Lambda_{jj}$
\STATE  Compute interval matrices $\mathbf{I}_{V}$ and $\mathbf{I}_{W}$ containing $V^{-1}$ and $W^{-1}$, resp. \COMMENT{For instance, using \texttt{verifylss.m} from INTLAB.} If this fails, or if $D$ has any zero elements, return \textbf{failure}
\STATE  $\mathbf{\check{X}} = \interval{\check{X},0}$ \COMMENT{To ensure that operations involving $\check{X}$ are performed in a verified fashion with interval arithmetic}
\STATE	$\mathbf{F}=A^*\mathbf{\check{X}}+Q+\mathbf{\check{X}}(A-G\mathbf{\check{X}})$\label{factorization}
\STATE	$\hat{\mathbf{F}}=\mathbf{I}_{W}^* \mathbf{F}V$
\STATE	$\hat{\mathbf{L}}=-\hat{\mathbf{F}}./D$\label{Hadamard3}
\STATE	$\hat{\Z}=\hat{\mathbf{L}}$
\FOR    {$k=1,\dotsc,k_{max}$}
\STATE	Set $\hat{\Z}=\square (0,\hat{\Z}\cdot \interval{1,0.1} + \interval {0,\texttt{realmin}}$ \COMMENT{$\varepsilon$-inflation technique}
\STATE	Compute $\hat{\mathbf{K}}$ using Algorithm~\ref{algo:krawczykophat} (or Algorithm~\ref{algo:newsuperset})
\IF {$\hat{\mathbf{K}} \subset \operatorname{int} (\hat{\Z})$ \{successful inclusion\}} 
\STATE Return $\X=\check{X}+W^*\hat{\mathbf{K}}\mathbf{I}_{V}$
\ENDIF
\STATE	$\hat{\Z}=\hat{\mathbf{K}}$
\ENDFOR
\STATE Return \textbf{failure} \COMMENT{Maximum number of iterations reached}
\end{algorithmic}
\end{algorithm}

\begin{algorithm}[h]
\small
\caption{Evaluating $\hat{\mathbf{K}}$ with $\va ({\hat{\mathbf{K}}})=\mathbf{k}_{\hat{f}}(\hat{\check{x}},\hat{R},\hat{\mathbf{z}},\hat{\mathbf{S}})$ encloses $\mathcal{K}_{\hat{f}}(\hat{\check{x}},\hat{R},\hat{\mathbf{z}},\hat{\mathcal{S}})$.}
\label{algo:krawczykophat}
\begin{algorithmic}[1]
\STATE   Input $A$, $G$, $Q$, $\check{X}$, $\hat{\Z}$\\
\COMMENT{Additionally, in this sub-function we use $V,W,\mathbf{I}_V,\mathbf{I}_W,\Lambda,D,\hat{\mathbf{L}}$ which are already computed in Algorithm~\ref{algo:mainenclosurehat}}
\STATE	$\hat{\mathbf{M}}=W^* \hat{\Z}\mathbf{I}_{V}$
\STATE	$\hat{\mathbf{N}}=\mathbf{I}_{W}^* (A-G(\check{X}+\hat{\mathbf{M}}))^*W^*$ \label{supersetlinehat}
\STATE	$\hat{\mathbf{O}}=\mathbf{I}_{V}(A-G(\check{X}+\hat{\mathbf{M}}))V$
\STATE	$\hat{\mathbf{P}}=(\Lambda^*-\hat{\mathbf{N}})\hat{\mathbf{M}}+\hat{\mathbf{M}}(\Lambda-\hat{\mathbf{O}})$
\STATE	$\hat{\mathbf{U}}=\hat{\mathbf{P}}./D$\label{Hadamard4}
\STATE  $\hat{\mathbf{K}}=\hat{\mathbf{L}}+\hat{\mathbf{U}}$
\STATE  Return $\hat{\mathbf{K}}$
\end{algorithmic}
\end{algorithm}

Note that computing $\hat{\mathbf{L}}$ and $\hat{\mathbf{U}}$ requires fewer dense $n\times n$ interval matrix multiplications than computing $\mathbf{L}$ and $\mathbf{U}$, so the impact of the wrapping effect is reduced. This is the reason why one expects Algorithm~\ref{algo:mainenclosurehat} to work in more cases than Algorithm~\ref{algo:mainenclosure}.

An important observation is that the last transformation $\X=\check{X}+W^*\hat{\mathbf{K}}\mathbf{I}_{V}$ happens after the Krawczyk verification procedure. So, while the procedure guarantees that only one zero $\hat{x}_s$ of $\hat{f}$ is contained in $W^{-*}\check{X}V+\hat{\mathbf{K}}$, when we return to the original setting and compute an enclosure for $\X=\check{X}+W^*\hat{\mathbf{K}}\mathbf{I}_{V}$, other solutions of~\eqref{care} may fall into this enclosure. Hence, Algorithm~\ref{algo:mainenclosurehat} alone does \emph{not} guarantee that there is a unique solution of~\eqref{care} in $\X$, \emph{nor} that this solution is the stabilizing one. We resolve with this issue in Section~\ref{sec:uniqueness}. 

Another small improvement introduced in this algorithm is gathering $\mathbf{\check{X}}$ in Line~\ref{factorization} of Algorithm~\ref{algo:mainenclosurehat}, in order to reduce the wrapping effect.

\subsection{Verifying a different Riccati equation}\label{verifyingdifferent}
Another possible modification to the verification process consists in modifying the equation into one with (possibly) better numerical properties. The idea stems from the formulation~\eqref{invsubformulation} of a CARE as an invariant subspace problem. We start from the following result.
\begin{lem} \label{invsublemma}
\cite{BiniIannazzoMeini2012} The stabilizing solution $X_s$ of CARE \eqref{care} is the only matrix $X\in\mathbb{C}^{n\times n}$ such that
\begin{equation} \label{invsub}
H \begin{bmatrix}
	I_n\\X_s
\end{bmatrix}
 = 
 \begin{bmatrix}
	I_n\\X_s
\end{bmatrix}R, \quad 
H = \begin{bmatrix}
	A & -G\\ -Q & -A^*
\end{bmatrix} \in \mathbb{C}^{2n\times 2n}
\end{equation}
for some Hurwitz stable matrix $R$. Moreover, it holds that $R=A-GX_s$.
\end{lem}
The subspace $\operatorname{im} \begin{bmatrix}
	I_n\\X_s
\end{bmatrix}$
is called \emph{stable invariant subspace} of the matrix $H$. We use this formulation to relate the solution $X_s$ to the one of a different CARE. The following result is a natural result of the literature on algebraic Riccati equations (see e.g.~\cite{BiniIannazzoMeini2012}), and the idea used here is certainly not original, but we prove it explicitly because we do not have a reference with this exact statement. 
\begin{lem} \label{changecare}
Let $X_s$ be the stabilizing solution of~\eqref{care}. Suppose that $P\in\mathbb{C}^{2n\times 2n}$ be a nonsingular matrix such that $P^{-1}HP$ has the same structure as $H$, i.e.,
\begin{equation} \label{hamiltonianstructure}
P^{-1}HP = \begin{bmatrix}
	A_P & -G_P\\ -Q_P & -A_P^*
\end{bmatrix} \in \mathbb{C}^{2n \times 2n},
\end{equation}
for some matrices $A_P,G_P=G_P^*,Q_P=Q_P^* \in \mathbb{C}^{n\times n}$. Let $Y_s$ be the stabilizing solution of the CARE 
\begin{equation} \label{carep}
 A_P^*Y+YA_P+Q_P=YG_PY,
\end{equation}
and $U_1,U_2\in\mathbb{C}^{n\times n}$ be defined by
\[
P\begin{bmatrix}
	I_n\\Y_s	
\end{bmatrix}
=
\begin{bmatrix}
	U_1\\U_2	
\end{bmatrix}.
\]
If $U_1$ is invertible, then $X_s=U_2U_1^{-1}$.
\end{lem}
\begin{proof}
We have
\[
P^{-1}HP \begin{bmatrix}
	I_n\\Y_s	
\end{bmatrix}
=
\begin{bmatrix}
	I_n\\Y_s	
\end{bmatrix}
R_P, 
\]
for the Hurwitz stable matrix $R_P=A_P-G_PY_s$. Multiplying both sides by $P$ on the left we get
\[
H \begin{bmatrix}
	U_1\\U_2	
\end{bmatrix}
=
\begin{bmatrix}
	U_1\\U_2	
\end{bmatrix}
R_P,
\]
and then multiplying on the right by $U_1^{-1}$
\[
H \begin{bmatrix}
	I_n\\U_2U_1^{-1}	
\end{bmatrix}
=
\begin{bmatrix}
	I_n\\U_2U_1^{-1}	
\end{bmatrix}
U_1R_PU_1^{-1}.
\]
Since $U_1R_PU_1^{-1}$ is Hurwitz stable, Lemma~\ref{invsublemma} gives us the thesis.
\end{proof}

The paper~\cite{MehrmannPoloni2012} contains a convenient strategy to construct a matrix $P$ with a particularly simple form (a permutation matrix with some sign changes) for which all the required assumptions hold and in addition $Y_s$ is bounded. Define for each $k=1,2,\dots,n$
\[
S_k:= \begin{bmatrix}
I_{n}-E_{kk} & E_{kk}\\
-E_{kk} & I_{n} - E_{kk}
\end{bmatrix} \in \mathbb{C}^{2n \times {2n}},
\]
where $E_{kk}$ is the matrix which has $1$ in position $(k,k)$ and zeros elsewhere; in other words, $S_k$ swaps the entries $k$ and $n+k$ of a vector in $\mathbb{C}^{2n}$, and changes sign to one of them. The matrices $S_k$ are orthogonal and commute with each other.
\begin{thm} \label{thm:mehp}\protect{~\cite[Theorem~3.4]{MehrmannPoloni2012}}
Let $\mathcal{I}=\{i_1,i_2,\dots,i_k\}$ be a subset of $\{1,2,\dots,n\}$, and $P= S_{i_1} S_{i_2}\dotsm S_{i_k}$. Then
\begin{enumerate}
	\item For each choice of $\mathcal{I}$, the matrix $P^{-1}HP$ has the structure~\eqref{hamiltonianstructure}.
	\item For each $\tau\geq \sqrt{2}$, one can find $\mathcal{I}$ such that $U_1$ is nonsingular and $Y_s$ has all its elements bounded in modulus by $\tau$ (referring to the definitions of $U_1$ and $Y_s$ in Lemma~\ref{changecare}).
\end{enumerate}
\end{thm}
These results suggest an alternative verification strategy:
\begin{enumerate}
	\item Compute $P$ satisfying Theorem~\ref{thm:mehp}.
	\item Form the coefficients $A_P,G_P$ and $Q_P$, which can be obtained from the entries of $H$ only using permutations and sign changes.\label{formingH-P}
	\item Using one of the various verification methods for CAREs, compute an interval $\Y$ containing the stabilizing solution $Y_s$.
	\item Compute
	\[
	\begin{bmatrix}
		\mathbf{U}_1\\
		\mathbf{U}_2
	\end{bmatrix} = P
	\begin{bmatrix}
		I_n\\ \Y
	\end{bmatrix},
	\]
which, again, requires only rearranging the entries and changing their signs, and hence can be done without wrapping effects.
	\item Compute using interval arithmetic a solution $\X$ to the linear system $\X \mathbf{U}_1 = \mathbf{U}_2$.
\end{enumerate}
Then, clearly, $\X$ contains the true stabilizing solution $X_s$ of~\eqref{care}. Again, since the interval matrix $\X$ computed in the last step is only a solution enclosure and suffers from wrapping effect, it might be the case that other solutions of the CARE~\eqref{care} are contained in $\X$ in addition to $X_s$.

The MATLAB toolbox~\cite{pgdoubling2012} contains algorithms to compute a subset $\mathcal{I}$ (and hence a matrix $P$) satisfying the conditions of Theorem~\ref{thm:mehp}, for every $\tau > \sqrt{2}$, in time bounded by $\mathcal{O}(n^3 \log_\tau n)$. The factor $\log_\tau n$ is a worst-case factor only, and in our experience for most matrices fine-tuning the choice of $\tau$ does not have a big impact on neither performance nor stability. Here, we always use the method with its default value $\tau=3$.

With this method, one transforms the problem of verifying~\eqref{care} into the one of verifying~\eqref{carep}; this latter Riccati equation has a stabilizing solution $Y_s$ whose entries are bounded in modulus by $\tau$, hence one may expect that less cancellation can take place in the algorithms. While there is no formal guarantee that this must happen, in practice, in most cases  the eigenvector matrix $V_P$ of $R_P=A_P-G_PY_s$ has a lower condition number than the one $V$ of $A-GX_s$, as we report in the experiments (see Table~\ref{fig:condcond} in the following), and verification of~\eqref{carep} is often easier than verification of~\eqref{care}. Ultimately, this is only a heuristic approach, though.

Let us analyze the computational complexity of Algorithm~\ref{algo:CAREverification}.
\begin{thm}\label{complexityalgo3}
Algorithm~\ref{algo:CAREverification} requires at most $\mathcal{O}(n^3 \log_\tau n + n^3s)$ floating point operations, where $s$ is the number of steps required by the inner verification algorithm in Line~\ref{line:innerverification}.
\end{thm}
\begin{proof}
Computing $\check{X}$ in Line~\ref{lst:caresol} requires $\mathcal{O}(n^3)$ operations, using for instance the algorithm mentioned in~\cite{Miyajima2015} (based on the ordered Schur form of $H$ and an additional Newton step with the residual computation performed in emulated quadruple-precision arithmetic). Forming $P$ in Theorem~\ref{thm:mehp} via the approach explored in~\cite{pgdoubling2012} costs $\mathcal{O}(n^3 \log_\tau n)$ floating point operations. Computing $\Y$ by using Algorithm~\ref{algo:mainenclosurehat} has cost $\mathcal{O}(n^3)$ per step (and the same will hold for Algorithm~\ref{algo:fixedpoint} that we will introduce later): the cost for the eigendecomposition and the enclosures $\mathbf{I}_V$ and $\mathbf{I}_W$ is again cubic in $n$, and all the other matrix-matrix operations (including the Hadamard divisions) in Algorithms~\ref{algo:mainenclosurehat} and~\ref{algo:newsuperset} have again cost $\mathcal{O}(n^3)$ at most, as they only involve $n \times n$ matrices.
\end{proof}


\subsection{A new superset}\label{newsuperset}
According to Theorem~\ref{modifiedkrawczyk}, the computed interval matrix is guaranteed to contain a unique solution if the set $\mathbf{S}$ contains the slopes $S(f;y,y')$  for all $y,y' \in {\x}$. On the other hand, if we employ an interval matrix containing only the slopes $S(f;\check{x},y')$ for all $y' \in {\x}$, existence can be proved, but not uniqueness. Since we have already decided to forgo (for now) uniqueness, it makes sense to let go of it also when choosing the superset $\mathbf{S}$. 

A simple modification to our proof of Lemma~\ref{scan2012} gives a tighter inclusion for the slope superset by reducing the wrapping effect.
\begin{thm}\label{anewsuperset}
Let $f$ be as in~\eqref{deff}, $\X\in\mathbb{IC}^{n\times n}$ be an interval matrix, and $\check{X}\in \X$ be Hermitian. Then, the interval matrix
\[
I_n\otimes (A-G\check{X})^*+(A-G\X)^T\otimes I_n
\]
contains the slopes $S(f, \check{x}, y')$ for each $y' \in \X$ where $\check{x}=\va (\check{X})$ and $y'=\va (Y').$
\end{thm}
\begin{proof}
We repeat the proof of Lemma~\ref{scan2012}, with $y=\check{x}$, and replace the term $\mathbf{K}_F(\X)$ in~\eqref{herminclusion} with the tighter inclusion $(I_n\otimes (A-G \check{X})^* + (A-G\X)^T\otimes I_n)$.
\end{proof}

As a consequence of Theorem~\ref{anewsuperset}, we can replace~\eqref{slopehat} with
\begin{equation}\label{betterslopehat}
\hat{\mathbf{S}} = I_n\otimes (W(A-G\check{X}) W^{-1})^*+(V^{-1} (A-G\mathbf{\check{X}})V)^T\otimes I_n
\end{equation}
in our modified Krawczyk algorithm applied to $\hat{f}$, and it will still yield an interval matrix containing a (possibly non-unique) solution of~\eqref{care}.

\begin{algorithm}[h]
\caption{Evaluating $\hat{\mathbf{K}}$ with $\va ({\hat{\mathbf{K}}})=\mathbf{k}_{\hat{f}}(\hat{\check{x}},\hat{R},\hat{\mathbf{z}},\hat{\mathbf{S}})$ encloses $\mathcal{K}_{\hat{f}}(\hat{\check{x}},\hat{R},\hat{\mathbf{z}},\hat{\mathcal{S}})$ with a tighter superset that does not guarantee solution uniqueness.} \label{algo:newsuperset}
\begin{algorithmic}[2]
\STATE \makeatletter\stepcounter{ALC@line}\makeatother 
\COMMENT{This algorithm is identical to Algorithm~\ref{algo:krawczykophat}, apart from Line~\ref{supersetlinehat} which is replaced by the following}
\STATE $\hat{\mathbf{N}}=\mathbf{I}_{W}^* (A-G\check{X})^*W^*$
\end{algorithmic}
\end{algorithm}

\begin{algorithm}[h]
\caption{Computation of an interval matrix $\X$ containing a solution of~\eqref{care} using permuted Riccati bases.} 
\label{algo:CAREverification}
\begin{algorithmic}[1]
\STATE	Input $A$, $G$, $Q$
\STATE	Compute an approximate stabilizing solution $\check{X}$ of~\eqref{care} using any floating point algorithm \label{lst:caresol}
\STATE	Compute a matrix $P$ satisfying point~\ref{formingH-P} of Theorem~\ref{thm:mehp} \COMMENT{For instance with the toolbox~\cite{pgdoubling2012}}

\STATE	Compute $A_P,G_P,Q_P$ satisfying~\eqref{hamiltonianstructure}
\STATE	Compute a verified solution $\Y$ to~\eqref{carep} using either Algorithm~\ref{algo:mainenclosurehat} or Algorithm~\ref{algo:fixedpoint} \label{line:innerverification}. If the verification fails, return \textbf{failure}
\STATE Set $\begin{bmatrix} {\mathbf{U}}_1 \\ {\mathbf{U}}_2
        \end{bmatrix}=P \begin{bmatrix}
        I_n \\ \Y
        \end{bmatrix}$
\STATE Compute an enclosure $\X$ for the solution of the interval system $\X{\mathbf{U}}_1 = {\mathbf{U}}_2$ \COMMENT{For instance, using \texttt{verifylss} from INTLAB}. If this fails, return \textbf{failure} 
\STATE Return $\X$
\end{algorithmic}
\end{algorithm}

\subsection{Verification of uniqueness and stabilizability} \label{sec:uniqueness}
As noted before, the modifications to the Krawczyk method introduced here do not ensure that the found interval matrix contains only one solution to~\eqref{care}. However, the following result holds.
\begin{thm}[\protect{\cite[Theorem~23.3]{Brockett_book}}]
The CARE~\eqref{care} has at most one stabilizing solution.
\end{thm}
A proof using the facts in~\cite{BiniIannazzoMeini2012} can be obtained by considering the eigenvalues of $H$. Let $X_s$ be a stabilizing solution, and let $\lambda_1,\lambda_2,\dots,\lambda_n$ be the eigenvalues of $A-GX_s$. Because of the formula~\eqref{invsubformulation}, $\lambda_1,\lambda_2,\dots,\lambda_n$ are also eigenvalues of $H$ (see \cite[Section~2.1.1]{BiniIannazzoMeini2012}). Moreover, the eigenvalues of $H$ have Hamiltonian symmetry, (see \cite[Section~1.5]{BiniIannazzoMeini2012}), so there are $n$ more eigenvalues with positive real part. We have identified $2n$ eigenvalues of $H$, counted with multiplicity, and none of them is purely imaginary; hence $H$ has no purely imaginary eigenvalue and \cite[Theorem~2.17]{BiniIannazzoMeini2012} holds.

Hence, if all matrices contained in $A-G\X$ are Hurwitz stable, then the solution $X\in \X$ is verified to be stabilizing (hence Hermitian) \emph{and} unique. To verify stability, we can use the method described in~\cite{Miy14}, which is summarized in~\cite[Lemma~2.4]{Miyajima2015}. The resulting method is described in Algorithm~\ref{algo:stabilizing}. In the algorithm, we use the notation $\Re z$ to mean the real part of the complex number $z$.
\begin{algorithm}[h]
\small
\caption{Verifying the Hurwitz stability of an interval matrix $\mathbf{M}$.} \label{algo:stabilizing}
\begin{algorithmic}[1]
\STATE Input $\mathbf{M}$
\STATE Compute approximations $V$, $W$, $\Lambda$ for the eigendecomposition of
 $\mi(\mathbf{M})$ in floating point \COMMENT{For instance, using the MATLAB command \texttt{eig}}
\STATE $\mathbf{V} = \interval{V,0}$
\STATE $R = \operatorname{mag}(W(\mathbf{M}\mathbf{V}-\mathbf{V}\Lambda))$
\STATE $S = \operatorname{mag} (I_n - \mathbf{V}W)$
\STATE $e = $ the $n\times 1$ matrix with $e_{i,1}=1$ for each $i$
\STATE $u = Re$
\STATE $t = Se$
\STATE $\mu = \max(u ./ (e-t))$
\STATE $r = u + \mu t$
\IF {($\max (t) < {1}$ and $r + \max (\Re (\operatorname{diag} (\Lambda)))e  < 0$)}
\STATE Return \textbf{success} \COMMENT{Every matrix $M\in\mathbf{M}$ is Hurwitz stable}
\ELSE 
\STATE Return \textbf{failure}
\ENDIF
\end{algorithmic}
\end{algorithm}

Notice one subtle point: when we apply Algorithm~\ref{algo:stabilizing} to $A-G\X$ we recompute $V$, $\Lambda$ and $W$ from the eigendecomposition of $\mi (A-G\X)$; this differs slightly from using the values computed previously, which came from the eigendecomposition of $A-G\check{X}$ (because $\X$ was not available at that point). This choice gives better results in our experiments. The cost for this verification is again $\mathcal{O}(n^3)$ floating point operations.

Hence, if the verification in Algorithm~\ref{algo:stabilizing} succeeds for the solution enclosure $\X$ returned by either Algorithm~\ref{algo:mainenclosure} or Algorithm~\ref{algo:mainenclosurehat}, then $\X$ contains exactly one solution of~\eqref{care}, and it is the stabilizing one.


\section{A direct fixed-point method}\label{directfixedpoint}
While the methods described in the previous sections work for many examples of Riccati equations, an essential limitation is that all of them require the closed-loop matrix $A-G\check{X}$ to be diagonalizable. Products with the eigenvector matrix $V$ and its inverse are required along the algorithm, and if these are ill-conditioned then the wrapping effects are more pronounced and the required inclusion $\mathbf{K} \subset \operatorname{int} (\mathbf{Z})$ or $\hat{\mathbf{K}} \subset \operatorname{int} (\hat{\mathbf{Z}})$ is less likely to hold. A striking example of this phenomenon is the first example in the benchmark set~\cite{carex1995}. This is a simple $2\times 2$ problem which appears in~\cite{carex1995} as nothing more than a ``warm-up example'', and yet all the verification methods described here (including those from~\cite{Hashemi2012} and~\cite{Miyajima2015}) fail.

To solve this issue, we would like to propose a different method for verification. The procedure is based on some ideas which appear in the context of ADI methods~\cite{Wachspress1988}. While this method is somehow more primitive and works on a lower number of examples, it does not require that the closed-loop matrix be diagonalizable.

We rewrite the CARE~\eqref{care} as follows. Given any Hermitian $\check{X}\in\mathbb{C}^{n \times n}$, one can write the exact stabilizing solution $X_s$ of the CARE~\eqref{care} as $X_s=\check{X}+Z_s$ for an unknown Hermitian correction matrix $Z_s$, and rewrite~\eqref{care} as a Riccati equation in $Z$,
\begin{equation} \label{shiftedcare}
\tilde{A}^*Z + Z\tilde{A} + \tilde{Q} =ZGZ, \quad \text{with } \tilde{A}=A-G\check{X},~\tilde{Q}=A^*\check{X} + \check{X} A + Q - \check{X} G \check{X}.
\end{equation}
The stabilizing solution of this equation is $Z_s$, since $\tilde{A}-GZ_s=A-GX_s$ is Hurwitz stable. Note that the degree-two coefficient $G$ is unchanged. For any $s\in\mathbb{C}$ such that $\tilde{A}-sI_n$ is nonsingular, \eqref{shiftedcare} is equivalent to the fixed point equation
\[
 Z = (\tilde{A}-sI_n)^{-*}(ZGZ-\tilde{Q}-Z(\tilde{A}+sI_n)).
\]
Thus, if we find an interval $\Z$ such that $(\tilde{A}-sI_n)^{-*}(\Z G\Z-\tilde{Q}-\Z(\tilde{A}+sI_n)) \subseteq \Z$, it follows from the Brouwer fixed-point theorem that~\eqref{shiftedcare} has a solution $Z_*\in\Z$, and that~\eqref{care} has a solution $X_* \in \check{X}+\Z$.

This simple iterative method is effective when $(\tilde{A}-sI_n)^{-*}\Z(\tilde{A}+sI_n)$ does not suffer excessively from wrapping effects, since we can expect $\tilde{Q}$ and the quadratic term $\Z G \Z$ to be small.

Are there any preconditioning transformations that we can make to improve the method? A possibility is applying a change of basis to the whole problem. Let $V\in\mathbb{C}^{n\times n}$ be invertible; we set
\begin{equation}
\label{defVmatrices}
Z_V = V^*ZV, \quad A_V = V^{-1}\tilde{A}V,~Q_V = V^*\tilde{Q}V,~G_V = V^{-1}GV^{-*},\end{equation}

so that~\eqref{shiftedcare} is transformed into
\begin{equation*}
	 A_V^*Z_V + Z_V A_V + Q_V = Z_V G_V Z_V. 
\end{equation*}
Continuing as above, we obtain the fixed-point equation
\begin{equation} \label{fixedpoint}
Z_V = (A_V-sI_n)^{-*}(Z_VG_VZ_V-Q_V-Z_V(A_V+sI_n)).	
\end{equation}
If $\tilde{A}$ is diagonalizable, we can set $V$ as its computed approximate eigenvector matrix, as in~\eqref{defVW}. One can see then that the resulting method has several steps in common with the Krawczyk method described in the previous sections. This time, though, we are free to choose the matrix $V$ without the risk of our method turning into a $\mathcal{O}(n^6)$ one, since everything in~\eqref{fixedpoint} is computable explicitly with standard linear algebra operations. 

Some heuristic experimentation led us to the following choices: we take $s$ equal to the approximation of $-\min\{\Re \lambda : \text{$\lambda$ is an eigenvalue of $\tilde{A}$}\}$ computed in floating-point arithmetic (motivated by the idea to make $A_V+sI_n$ small and $A_V-sI_n$ large), and $V$ as the orthogonal factor of the (computed) Schur factorization of $\tilde{A} \approx VTV^{-1}$ (motivated by the idea to concentrate most of the ``weight'' of $V^{-1}\tilde{A}V$ on its diagonal). The matrix $\tilde{A}$ is an approximation of $A-GX_s$, which is Hurwitz stable, so in exact arithmetic we would have $s>0$ and $A_V-sI_n = V^{-1}(\tilde{A}-sI_n)V$ invertible, since its eigenvalues are $\lambda_i-s$, where $\lambda_i$ are the eigenvalues of $A-GX_s$, and thus have strictly negative real part. Hence these properties are likely to hold also for its computed approximation $\tilde{A}$.

The resulting algorithm is described in Algorithm~\ref{algo:fixedpoint}.
\begin{algorithm}[h]
\caption{Computation of an interval matrix $\X$ containing a solution of~\eqref{care} using a simple fixed-point algorithm.} \label{algo:fixedpoint}
\begin{algorithmic}[1]
\STATE	Input $A$, $G$, $Q$
\STATE	Compute an approximate stabilizing solution $\check{X}$ of~\eqref{care} using any floating point algorithm
\STATE Compute $\tilde{A}$ (in floating point) as in~\eqref{shiftedcare}
\STATE Choose $s$ and $V$, for instance $s \approx -\min\{\Re \lambda : \lambda \text{ is an eigenvalue of $\tilde{A}$}\}$ and $V$ as the (approximate) orthogonal Schur factor of $\tilde{A}$
\STATE Compute an interval matrix $\mathbf{I}_V$ containing $V^{-1}$ \COMMENT{For instance, using \texttt{verifylss} from INTLAB}
\STATE Compute interval matrices $\A_V,\mathbf{G}_V$, $\mathbf{Q}_V$ containing $A_V$, $G_V$, $Q_V$, respectively \COMMENT{Replacing $\check{X}$ and $V$ in~\eqref{shiftedcare} and~\eqref{defVmatrices} with $\mathbf{\check{X}}=\interval{\check{X},0}$ and $\mathbf{V} = \interval{V,0}$, respectively}
\STATE Compute an interval matrix $\mathbf{I}_s$ containing $(\A_V^* - sI_n)^{-1}$ \COMMENT{For instance, using \texttt{verifylss} from INTLAB}
\STATE Set $k=0$ and $\Z_V = -\mathbf{I}_s \mathbf{Q}_V$
\FOR{$k=1,\dots,k_{max}$}
\STATE	Set $\Z_V = \square (0,\Z_V \cdot \interval{1,0.1} +\interval{0,\texttt{realmin}})$ \COMMENT{$\varepsilon-$inflation technique}
	\STATE Set $\Y = \mathbf{I}_s(-\mathbf{Q}_V - \Z_V (\A_V + sI_n - \mathbf{G}_V\Z_V))$
\IF{$\Y \subset \operatorname{int} (\Z_V)$ }
\STATE Return $\X = \check{X} + \mathbf{I}_V^* \Z_V \mathbf{I}_V$
\ENDIF
\ENDFOR
\STATE Return \textbf{failure} \COMMENT{Maximum number of iterations reached}
\end{algorithmic}
\end{algorithm}
\begin{thm}\label{complexityalgo4}
Algorithm~\ref{algo:fixedpoint} has a cost of $\mathcal{O}(n^3s)$ arithmetic operations, if the verification succeeds in $s$ steps.
\end{thm}
\begin{proof}
Again, all the required operations in every step are matrix-matrix operations between $n\times n$ matrices. The Schur decomposition requires $\mathcal{O}(n^3)$ operations as well, in practice~\cite{GolubVanLoan2013}.
\end{proof}

Once again, uniqueness is not guaranteed, but it can be deduced \emph{a posteriori} if the verification of the stabilizing property of the computed inclusion interval $\X$ succeeds.

\section{Numerical experiments}\label{numerical}
This section presents numerical experiments to validate the algorithms. We compare four different approaches: 
\begin{enumerate}
	\item The modified Krawczyk approach described in~\cite{Hashemi2012} and in Section~\ref{residualKrawczykapproach}. This corresponds to Algorithm~\ref{algo:mainenclosure}. When the algorithm is successful, we check afterwards whether $A-G\X$ is Hurwitz stable using Algorithm~\ref{algo:stabilizing}.
	We call this approach \emph{Method H} in the following.
	\item The method described in~\cite{Miyajima2015} (using the MATLAB implementation \texttt{Mn.m} published by its author). The method already includes running Algorithm~\ref{algo:stabilizing} to check if the computed solution is Hurwitz stable, so we do not need any additional steps. We call this procedure \emph{Method M}.
	\item Algorithm~\ref{algo:CAREverification}, choosing as its subroutine to solve the transformed CARE~\eqref{carep} the Krawczyk-based Algorithm~\ref{algo:mainenclosurehat} and the modified superset trick used in Algorithm~\ref{algo:newsuperset}. This is a combination of all the improvements to Method H described in Section~\ref{findingenclosures}. We call this procedure \emph{Method K} (where K stands for Krawczyk). When the algorithm is successful, we check afterwards whether $A-G\X$ is Hurwitz stable using Algorithm~\ref{algo:stabilizing}.
	\item Algorithm~\ref{algo:CAREverification} again, but using the fixed-point Algorithm~\ref{algo:fixedpoint} to solve the transformed CARE~\eqref{carep}. This is a combination of the techniques described in Sections~\ref{verifyingdifferent} and~\ref{directfixedpoint}.
	We call this procedure \emph{Method F} (where F stands for fixed-point). When the algorithm is successful, we check afterwards whether $A-G\X$ is Hurwitz stable using Algorithm~\ref{algo:stabilizing}.
\end{enumerate}
The algorithms were tested in MATLAB 2015b with INTLAB v6, using unit round off $u=2^{-53}\approx {1.1\times {10^{-16}}}$, and run on a computer with an Intel core Duo \@2.66GHz processor and 6GB main memory.

The required stabilizing solutions of CAREs are computed using the method described in~\cite{Miyajima2015} (ordered Schur method followed by one step of Newton refinement in simulated quadruple precision).

In order to assess the quality of the enclosures computed in each experiment we use the norm-wise relative error $\nre$ and the geometric average relative precision $\garp$. The first error measure is defined as 
\[\nre (\X):= \operatorname{mag}\frac{\left\|\ra({\X})\right\|_{F}}{\left\|\X\right\|_{F}}.\]
This is the simplest possible bound for the (norm-wise) relative error
\[
\frac{\norm{X_s-\mi (\X)}_{F}}{\norm{X_s}_{F}}
\]
obtained by taking $\mi (\X)$ as an approximation of the solution.

Following previous works~\cite[e.g.]{Frommer2012}, we also report a component-wise error indicator $\garp$ based on the relative precision of an interval, $\rp(\x)$, defined as
\[\rp(\x):=\min( \relerr(\x),1),\]
where $\relerr$ is the relative error of the interval $\x=\interval{\mi({\x}),\ra({\x})}$ defined by 
\[ 
\relerr (\x):= 
\begin{cases}
\left| \frac{ \ra ({\x})}{ \mi ({\x})}\right|,& \quad \text{if} \quad 0\notin \x, \\
\ra ({\x}),& \quad \text{otherwise}.
\end{cases}
\]
We define our residual measure as the geometric average of $\rp(\X_{ij})$
\[ \garp (\X):=\left(\prod_{i,j=1}^{n} \rp(\X_{ij})\right)^{\frac{1}{n^2}}, \quad \X=(\X_{ij}).\]
The quantity $-\log{(\rp(\x))}$ can be interpreted as the number of known correct digits of an \emph{exact} value contained in $\x$; so, loosely speaking, $-\log{(\garp(\X))}$ represents the average number of known correct digits~\cite{Frommer2014}.

\subsection{CAREX Benchmark problems}

We ran these algorithms on all the equations from the benchmark set described in~\cite{ChuLiuMehrmann2007}, which contains experiments taken from the test suite CAREX~\cite{carex1995}, run with both default and non-default arguments. The results are reported in Tables~\ref{comparison1}--\ref{stabilizingproperty}, and a visualization of the results is in Figure~\ref{fig:garps}.

The \texttt{Experiment number} follows the order used in~\cite{ChuLiuMehrmann2007}. Note that this set of problems is designed to be challenging for non-verified CARE solvers in machine arithmetic, so it is not surprising that the verification algorithms cannot deal with all of them with perfect accuracy.

When the algorithms are successful, we report in Tables~\ref{comparison1}--\ref{comparison3} the number $k$ of required iterations of the outer Krawczyk loop. If the algorithm breaks down or does not converge within the maximum number of steps (which we took to be 50 for the iterative Methods~H,~K and F), then we write a star in the corresponding column. Method~M is not iterative, therefore for it we put $-$ in the column containing the number of iterations.

The size of the problem (value of $n$) and the total time (in seconds, including the time required to verify the stabilizing property) taken on our test machine are reported, too, as well as the norm-2 condition number of $V$ (used by Methods~H,~M and~K) and the same quantity for the eigenvector matrix $V_P$ of the closed-loop matrix $A_P-G_PY_{s}$ used in the two Methods~K and~F. All these details are given in Tables~\ref{comparison1}--\ref{comparison3} in the column named $\texttt{Problem property}$.

\bgroup
\def\arraystretch{1.1}
\begin{landscape}
\begin{table}[h!]\small
\centering
\caption{Comparison between various proposed methods}\label{comparison1}
\begin{tabular}{|c||c|c||c|c||c|c||c|c||c|c|}
	\hline
	
\texttt{Experiment} & \multicolumn{2}{c||}{\texttt{Problem property}} & \multicolumn{2}{c||}{\texttt{Method~H}} & \multicolumn{2}{c||}{\texttt{Method~M}} & \multicolumn{2}{c||}{\texttt{Method~K}} & \multicolumn{2}{c|}{\texttt{Method~F}}\\
\cline{2-11}
	\texttt{number} & \multicolumn{2}{c||}{size} & nre & $k$ & nre & $k$ & nre & $k$ & nre & $k$ \\
	\cline{2-3}
	 in~\cite{ChuLiuMehrmann2007} & $\operatorname{cond}(V)$ & $\operatorname{cond}(V_P)$ & garp & time & garp & time & garp & time & garp & time \\
	
	\hline
	\hline
1 & \multicolumn{2}{c||}{2} &   NaN &  *  &   NaN &  -  &   NaN &  *  & 3.75e-15 & 2\\
\cline{2-3}
 & 7.75e+07 & 3.17e+07 &   NaN &  *  &   NaN &  *  &   NaN &  *  & 4.18e-15 & 2.93e-02\\
\hline
\hline
2 & \multicolumn{2}{c||}{2} & 9.67e-14 & 1 & 4.65e-15 &  -  & 1.21e-14 & 1 & 1.00e-14 & 3\\
\cline{2-3}
 & 1.01e+01 & 1.15e+00 & 1.04e-13 & 2.00e-02 & 4.97e-15 & 7.59e-03 & 1.27e-14 & 2.28e-02 & 1.06e-14 & 2.70e-02\\
\hline
\hline
3 & \multicolumn{2}{c||}{4} & 3.93e-14 & 1 & 2.99e-15 &  -  & 3.70e-14 & 1 & 8.04e-14 & 6\\
\cline{2-3}
 & 9.73e+00 & 5.11e+00 & 2.80e-14 & 2.28e-02 & 2.12e-15 & 1.03e-02 & 5.02e-14 & 2.90e-02 & 1.05e-13 & 4.28e-02\\
\hline
\hline
4 & \multicolumn{2}{c||}{8} & 1.02e-14 & 1 & 2.34e-15 &  -  & 7.76e-14 & 1 & 1.03e-13 & 15\\
\cline{2-3}
 & 1.23e+00 & 2.18e+00 & 1.49e-14 & 1.74e-02 & 3.42e-15 & 9.05e-03 & 1.05e-13 & 2.44e-02 & 1.38e-13 & 5.72e-02\\
\hline
\hline
5 & \multicolumn{2}{c||}{9} & 6.73e-14 & 1 & 1.10e-14 &  -  & 4.34e-13 & 1 & 2.06e-12 & 42\\
\cline{2-3}
 & 7.54e+01 & 6.52e+01 & 4.34e-14 & 1.78e-02 & 1.05e-14 & 9.37e-03 & 7.57e-13 & 2.45e-02 & 4.70e-12 & 1.23e-01\\
\hline
\hline
6 & \multicolumn{2}{c||}{30} & 4.79e-13 & 2 & 3.35e-14 &  -  & 9.20e-09 & 2 &   NaN &  * \\
\cline{2-3}
 & 1.11e+05 & 3.48e+03 & 2.92e-11 & 5.15e-02 & 1.87e-12 & 1.88e-02 & 1.15e-08 & 6.64e-02 &   NaN &  * \\
\hline
\hline
7 & \multicolumn{2}{c||}{2} & 2.35e-16 & 2 & 2.12e-16 &  -  & 5.57e-16 & 1 & 7.47e-16 & 3\\
\cline{2-3}
 & 1.62e+00 & 3.31e+00 & 5.48e-16 & 2.18e-02 & 4.36e-16 & 7.58e-03 & 6.17e-16 & 2.27e-02 & 8.72e-16 & 2.69e-02\\
\hline
\hline
8 & \multicolumn{2}{c||}{2} & 3.22e-08 & 1 & 1.78e-08 &  -  & 3.67e-16 & 1 & 8.55e-16 & 2\\
\cline{2-3}
 & 1.01e+00 & 2.42e+00 & 4.75e-10 & 1.68e-02 & 3.04e-10 & 1.08e-02 & 4.83e-16 & 2.34e-02 & 1.88e-10 & 2.52e-02\\
\hline
\hline
9 & \multicolumn{2}{c||}{2} & 6.41e-16 & 1 & 1.92e-16 &  -  & 3.34e-10 & 1 & 2.25e-09 & 7\\
\cline{2-3}
 & 1.22e+00 & 6.99e+01 & 2.59e-15 & 1.69e-02 & 4.87e-16 & 8.31e-03 & 3.36e-10 & 2.28e-02 & 2.26e-09 & 3.64e-02\\
\hline
\hline
10 & \multicolumn{2}{c||}{2} &   NaN &  *  & 5.36e-12 &  -  & 3.00e-08 & 1 & 1.59e-08 & 45\\
\cline{2-3}
 & 6.80e+01 & 1.11e+00 &   NaN &  *  & 5.36e-12 & 8.49e-03 & 3.00e-08 & 2.15e-02 & 1.59e-08 & 1.34e-01\\
\hline
\hline
11 & \multicolumn{2}{c||}{2} & 9.65e-16 & 1 & 6.29e-16 &  -  & 2.45e-15 & 1 & 4.53e-15 & 2\\
\cline{2-3}
 & 3.74e+00 & 1.01e+00 & 1.04e-15 & 2.20e-02 & 6.75e-16 & 7.71e-03 & 2.63e-15 & 2.90e-02 & 4.87e-15 & 3.00e-02\\
\hline
\end{tabular} 
\end{table}
\end{landscape}
\egroup
\newpage

\bgroup
\def\arraystretch{1.1}
\begin{landscape}
\begin{table}[h!]\small
\centering
\caption{Comparison between various proposed methods}\label{comparison2}
\begin{tabular}{|c||c|c||c|c||c|c||c|c||c|c|}
	\hline
	
\texttt{Experiment} & \multicolumn{2}{c||}{\texttt{Problem property}} & \multicolumn{2}{c||}{\texttt{Method~H}} & \multicolumn{2}{c||}{\texttt{Method~M}} & \multicolumn{2}{c||}{\texttt{Method~K}} & \multicolumn{2}{c|}{\texttt{Method~F}}\\
\cline{2-11}
	\texttt{number} & \multicolumn{2}{c||}{size} & nre & $k$ & nre & $k$ & nre & $k$ & nre & $k$ \\
	\cline{2-3}
	 in~\cite{ChuLiuMehrmann2007} & $\operatorname{cond}(V)$ & $\operatorname{cond}(V_P)$ & garp & time & garp & time & garp & time & garp & time \\
	\hline
	\hline
12 & \multicolumn{2}{c||}{2} & 7.96e-16 & 1 & 3.22e-16 &  -  & 1.53e-15 & 1 & 3.90e-11 & 2\\
\cline{2-3}
 & 1.42e+03 & 1.01e+00 & 8.94e-16 & 2.20e-02 & 3.74e-16 & 7.50e-03 & 2.29e-15 & 2.82e-02 & 6.56e-11 & 2.95e-02\\
\hline
\hline
13 & \multicolumn{2}{c||}{2} & 9.41e-09 & 1 & 2.23e-09 &  -  & 6.99e-16 & 1 &   NaN &  * \\
\cline{2-3}
 & 2.42e+00 & 1.01e+00 & 3.01e-10 & 1.69e-02 & 7.20e-11 & 8.29e-03 & 1.15e-15 & 2.28e-02 &   NaN &  * \\
\hline
\hline
14 & \multicolumn{2}{c||}{2} & 1.65e-15 & 1 & 2.69e-16 &  -  & 9.49e-16 & 1 & 4.37e-15 & 3\\
\cline{2-3}
 & 1.01e+00 & 1.01e+00 & 1.92e-15 & 1.65e-02 & 3.14e-16 & 7.24e-03 & 1.09e-15 & 2.29e-02 & 5.09e-15 & 2.71e-02\\
\hline
\hline
15 & \multicolumn{2}{c||}{2} & 4.72e-11 & 1 & 3.36e-12 &  -  & 2.92e-15 & 1 &   NaN &  * \\
\cline{2-3}
 & 1.01e+00 & 1.29e+00 & 4.72e-11 & 1.67e-02 & 3.36e-12 & 8.09e-03 & 2.39e-15 & 2.27e-02 &   NaN &  * \\
\hline
\hline
16 & \multicolumn{2}{c||}{2} &   NaN &  *  & 5.87e-10 &  -  & 9.38e-12 & 1 &   NaN &  * \\
\cline{2-3}
 & 1.00e+00 & 1.29e+00 &   NaN &  *  & 5.87e-10 & 8.07e-03 & 5.63e-12 & 2.15e-02 &   NaN &  * \\
\hline
\hline
17 & \multicolumn{2}{c||}{2} & 2.42e-15 & 1 & 2.23e-16 &  -  & 4.80e-15 & 1 & 1.25e-14 & 5\\
\cline{2-3}
 & 1.01e+00 & 2.62e+00 & 2.69e-15 & 2.23e-02 & 2.23e-16 & 9.29e-03 & 4.98e-15 & 2.83e-02 & 1.35e-14 & 3.92e-02\\
\hline
\hline
18 & \multicolumn{2}{c||}{2} &   NaN &  *  &   NaN &  -  &   NaN &  *  &   NaN &  * \\
\cline{2-3}
 & 1.01e+00 & 2.62e+00 &   NaN &  *  &   NaN &  *  &   NaN &  *  &   NaN &  * \\
\hline
\hline
19 & \multicolumn{2}{c||}{3} & 3.53e-15 & 1 & 2.73e-16 &  -  & 3.67e-15 & 1 & 6.64e-15 & 3\\
\cline{2-3}
 & 1.01e+00 & 1.01e+00 & 1.20e-14 & 1.98e-02 & 8.76e-16 & 7.88e-03 & 1.19e-14 & 2.29e-02 & 2.16e-14 & 2.72e-02\\
\hline
\hline
20 & \multicolumn{2}{c||}{3} & 9.95e-05 & 3 & 3.87e-05 &  -  & 4.77e-15 & 1 & 1.73e-14 & 3\\
\cline{2-3}
 & 1.01e+00 & 1.00e+00 & 1.18e-04 & 2.78e-02 & 4.53e-05 & 8.36e-03 & 6.08e-12 & 2.30e-02 & 2.08e-14 & 2.70e-02\\
\hline
\hline
21 & \multicolumn{2}{c||}{4} & 1.29e-14 & 1 & 4.76e-15 &  -  & 1.26e-13 & 1 & 3.81e-13 & 11\\
\cline{2-3}
 & 9.01e+00 & 3.58e+00 & 1.73e-14 & 2.25e-02 & 6.41e-15 & 8.79e-03 & 1.39e-13 & 2.92e-02 & 4.32e-13 & 5.84e-02\\
\hline
\hline
22 & \multicolumn{2}{c||}{4} & 7.28e-05 & 2 & 3.52e-06 &  -  & 1.70e-04 & 1 &   NaN &  * \\
\cline{2-3}
 & 1.22e+01 & 5.94e+00 & 4.81e-06 & 3.02e-02 & 2.96e-07 & 1.04e-02 & 8.34e-06 & 2.90e-02 &   NaN &  * \\
\hline\hline

\end{tabular} 
\end{table}
\end{landscape}
\egroup
\newpage

\bgroup
\def\arraystretch{1.1}
\begin{landscape}
\begin{table}[h!]\small
\centering
\caption{Comparison between various proposed methods}\label{comparison3}
\begin{tabular}{|c||c|c||c|c||c|c||c|c||c|c|}
	\hline
	
\texttt{Experiment} & \multicolumn{2}{c||}{\texttt{Problem property}} & \multicolumn{2}{c||}{\texttt{Method~H}} & \multicolumn{2}{c||}{\texttt{Method~M}} & \multicolumn{2}{c||}{\texttt{Method~K}} & \multicolumn{2}{c|}{\texttt{Method~F}}\\

\cline{2-11}
	\texttt{number} & \multicolumn{2}{c||}{size} & nre & $k$ & nre & $k$ & nre & $k$ & nre & $k$ \\
	\cline{2-3}
	 in~\cite{ChuLiuMehrmann2007} & $\operatorname{cond}(V)$ & $\operatorname{cond}(V_P)$ & garp & time & garp & time & garp & time & garp & time \\
	
	\hline
	\hline
23 & \multicolumn{2}{c||}{4} & 3.25e-14 & 1 & 1.78e-15 &  -  & 3.43e-14 & 1 & 1.30e-13 & 11\\
\cline{2-3}
 & 1.43e+01 & 1.77e+00 & 3.00e-14 & 2.23e-02 & 1.71e-15 & 7.96e-03 & 4.61e-14 & 2.92e-02 & 1.73e-13 & 5.84e-02\\
\hline
\hline
24 & \multicolumn{2}{c||}{4} &   NaN &  *  &   NaN &  -  &   NaN &  *  &   NaN &  * \\
\cline{2-3}
 & 1.74e+00 & 1.74e+00 &   NaN &  *  &   NaN &  *  &   NaN &  *  &   NaN &  * \\
\hline
\hline
25 & \multicolumn{2}{c||}{77} & 4.08e-12 & 1 & 3.66e-13 &  -  & 4.70e-11 & 1 & 3.16e-10 & 12\\
\cline{2-3}
 & 4.98e+01 & 1.87e+01 & 3.50e-11 & 2.11e-01 & 3.14e-12 & 1.13e-01 & 2.64e-10 & 3.03e-01 & 1.63e-09 & 4.34e-01\\
\hline
\hline
26 & \multicolumn{2}{c||}{237} & 4.75e-11 & 1 & 4.35e-12 &  -  & 2.27e-09 & 1 & 1.26e-08 & 17\\
\cline{2-3}
 & 2.41e+02 & 8.93e+01 & 8.89e-10 & 3.05e+00 & 8.21e-11 & 1.53e+00 & 1.41e-08 & 4.56e+00 & 7.09e-08 & 6.79e+00\\
\hline
\hline
27 & \multicolumn{2}{c||}{397} &   NaN &  *  & 6.71e-12 &  -  & 8.73e-09 & 1 & 6.69e-08 & 19\\
\cline{2-3}
 & 1.31e+02 & 4.84e+01 &   NaN &  *  & 2.27e-10 & 5.26e+00 & 5.50e-08 & 1.82e+01 & 3.87e-07 & 2.87e+01\\
\hline
\hline
28 & \multicolumn{2}{c||}{8} & 4.35e-15 & 1 & 1.67e-15 &  -  & 4.35e-15 & 1 & 1.30e-14 & 4\\
\cline{2-3}
 & 1.01e+00 & 1.01e+00 & 8.95e-15 & 1.74e-02 & 3.44e-15 & 7.72e-03 & 8.95e-15 & 2.37e-02 & 2.66e-14 & 3.04e-02\\
\hline
\hline
29 & \multicolumn{2}{c||}{64} & 4.12e-13 & 1 & 4.99e-14 &  -  & 4.12e-13 & 1 & 7.12e-13 & 4\\
\cline{2-3}
 & 1.01e+00 & 1.01e+00 & 1.97e-07 & 4.56e-02 & 2.38e-08 & 3.63e-02 & 1.97e-07 & 7.01e-02 & 3.40e-07 & 8.96e-02\\
\hline
\hline
30 & \multicolumn{2}{c||}{21} &   NaN &  *  &   NaN &  -  & 3.90e-04 & 1 &   NaN & 38\\
\cline{2-3}
 & 2.42e+09 & 2.78e+00 &   NaN &  *  &   NaN &  *  & 3.79e-04 & 4.34e-02 &   NaN & 2.08e-01\\
\hline
\hline
31 & \multicolumn{2}{c||}{21} &   NaN &  *  &   NaN &  -  &   NaN & 1 &   NaN &  * \\
\cline{2-3}
 & 2.42e+09 & 2.88e+02 &   NaN &  *  &   NaN &  *  &   NaN & 5.06e-02 &   NaN &  * \\
\hline
\hline
32 & \multicolumn{2}{c||}{100} & 6.57e-12 & 1 & 1.13e-12 &  -  & 6.57e-12 & 1 &   NaN &  * \\
\cline{2-3}
 & 1.01e+00 & 1.01e+00 & 2.27e-11 & 1.52e-01 & 3.90e-12 & 1.36e-01 & 2.27e-11 & 2.33e-01 &   NaN &  * \\
\hline
\hline
33 & \multicolumn{2}{c||}{60} & 2.04e-14 & 1 & 2.67e-13 &  -  & 2.77e-10 & 1 &   NaN &  * \\
\cline{2-3}
 & 1.91e+01 & 1.55e+01 & 4.67e-14 & 1.12e-01 & 6.10e-13 & 4.97e-02 & 4.17e-10 & 1.58e-01 &   NaN &  * \\
\hline
\end{tabular} 
\end{table}
\end{landscape}
\egroup

\pgfplotstableread{results.table}{\results}
\begin{figure}[h!]
\centering
\begin{tikzpicture}
\begin{loglogaxis}[ymin=1e-18,ymax=1e0, xlabel=$\operatorname{cond}(V)$, ylabel=$nre$]
\addplot+[only marks] table[x=condV,y=hashemi.nre] \results;
\end{loglogaxis}
\end{tikzpicture}
\caption{nre of Method~H vs. $\operatorname{cond}(V)$} \label{fig:condH}
\end{figure}

\pgfplotstableread{results.table}{\results}
\begin{figure}
\centering
\begin{tikzpicture}
\begin{loglogaxis}[ymin=1e-18,ymax=1e0,xlabel=$\operatorname{cond}(V_P)$, ylabel=$nre$]
\addplot+[only marks] table[x=condVp,y=miyajima.nre] \results;
\end{loglogaxis}
\end{tikzpicture}
\caption{nre of Method~M vs. $\operatorname{cond}(V_P)$} \label{fig:condK}
\end{figure}

\pgfplotstableread{results.table}{\results}
\begin{figure}
\centering
\begin{tikzpicture}
\begin{loglogaxis}[ymin=1e-18,ymax=1e0,xlabel=$\operatorname{cond}(V_P)$, ylabel=$nre$]
\addplot+[only marks] table[x=condV,y=our.nre] \results;
\end{loglogaxis}
\end{tikzpicture}
\caption{nre of Method~K vs. $\operatorname{cond}(V)$} \label{fig:condM}
\end{figure}

\pgfplotstableread{results.table}{\results}
\begin{figure}
\centering
\begin{tikzpicture}
\begin{loglogaxis}[ymin=1e-18,ymax=1e0,xlabel=$\operatorname{cond}(V_P)$, ylabel=$nre$]
\addplot+[only marks] table[x=condV,y=fixed.nre] \results;
\end{loglogaxis}
\end{tikzpicture}
\caption{nre of Method~F vs. $\operatorname{cond}(V)$} \label{fig:condF}
\end{figure}

\pgfplotstableread{results.table}{\results}
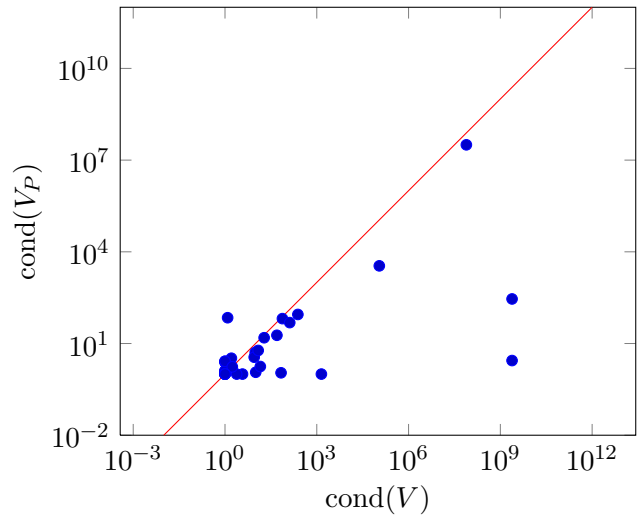
\begin{figure}
\centering
\begin{tikzpicture}
\begin{loglogaxis}[xlabel=$\operatorname{cond}(V)$, ylabel=$\operatorname{cond}(V_P)$, axis equal, enlargelimits = false]
\addplot+[only marks] table[x=condV,y=condVp] \results;
\addplot+[mark = none, domain = 1e-2:1e12] {x};
\end{loglogaxis}
\end{tikzpicture}
\caption{$\operatorname{cond}(V_P)$ vs. $\operatorname{cond}(V)$. Most of the points lie below the axes bisector (drawn in red), which means that the condition number of $V_P$ is generally lower than the one of $V$.} \label{fig:condcond}
\end{figure}

\pgfplotstableread{results.table}{\results}
\begin{landscape}
\begin{figure}[h!]
\centering
\begin{tikzpicture}
\begin{axis}[ybar=1pt, ymode=log,log origin=infty, width=1.6\textwidth, height=\textheight, bar width = 2pt,
   xtick={1,...,33},
   xmin=0.5,
   xmax=33.5, ymin = 1e-18, ymax=1e-2, legend pos = north west, xlabel = {experiment number}, ylabel={$garp$}]
\addplot table[x=experiment,y=hashemi.garp] \results; \addlegendentry{Method~H}
\addplot table[x=experiment,y=miyajima.garp] \results; \addlegendentry{Method~M}
\addplot table[x=experiment,y=our.garp] \results; \addlegendentry{Method~K} 
\addplot table[x=experiment,y=fixed.garp] \results; \addlegendentry{Method~F}

\end{axis}
\end{tikzpicture}
\caption{Values of garp for each experiment number. A full bar means that the method failed to compute an enclosure.} \label{fig:garps}
\end{figure}
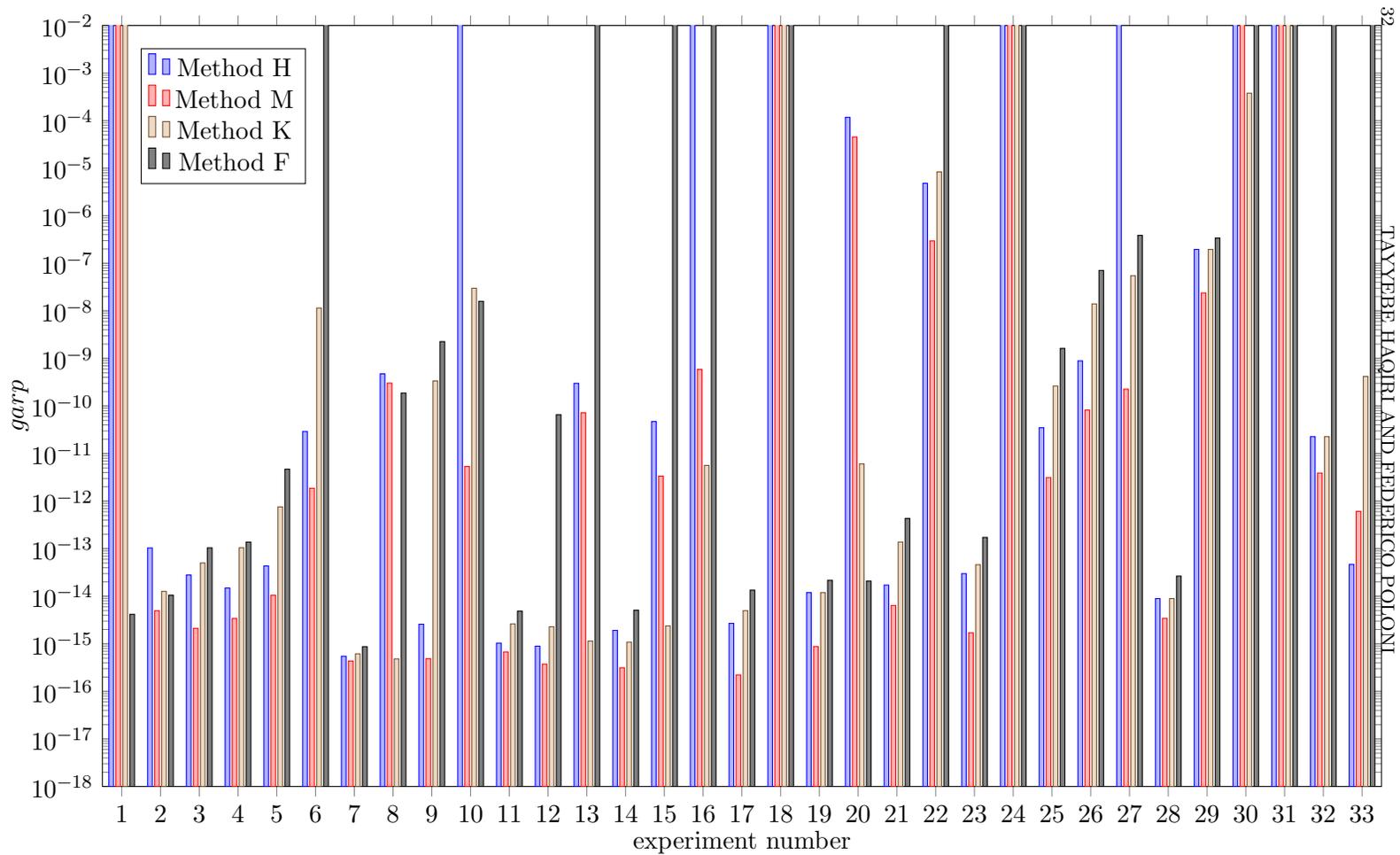
\end{landscape}

Table~\ref{stabilizingproperty} reports the result of checking the stabilization property; a plus sign means that the property is verified, a minus sign means failure to verify the property, and a star means that the algorithm had already failed to compute an inclusion interval. As one can see, there is only a very limited number of cases in which the stabilization procedure fails. 

\bgroup
\def\arraystretch{1.1}
\begin{table}[h!]
\centering
\caption{Results for stabilizing property in all methods}\label{stabilizingproperty}
\begin{tabular}{|c||c|c|c|c|}
	\hline
	
\texttt{Experiment number} & \texttt{Method~H} & \texttt{Method~M} & \texttt{Method~K} & \texttt{Method~F}\\

\cline{1-5}
1 &  *  &  *  &  *  &  + \\
\hline
2 &  +  &  +  &  +  &  + \\
\hline
3 &  +  &  +  &  +  &  + \\
\hline
4 &  +  &  +  &  +  &  + \\
\hline
5 &  +  &  +  &  +  &  + \\
\hline
6 &  +  &  +  &  +  &  * \\
\hline
7 &  +  &  +  &  +  &  + \\
\hline
8 &  +  &  +  &  +  &  + \\
\hline
9 &  +  &  +  &  +  &  + \\
\hline
10 &  *  &  +  &  -  &  - \\
\hline
11 &  +  &  +  &  +  &  + \\
\hline
12 &  +  &  +  &  +  &  + \\
\hline
13 &  +  &  +  &  +  &  * \\
\hline
14 &  +  &  +  &  +  &  + \\
\hline
15 &  +  &  +  &  +  &  * \\
\hline
16 &  *  &  +  &  +  &  * \\
\hline
17 &  +  &  +  &  +  &  + \\
\hline
18 &  *  &  *  &  *  &  * \\
\hline
19 &  +  &  +  &  +  &  + \\
\hline
20 &  +  &  +  &  +  &  + \\
\hline
21 &  +  &  +  &  +  &  + \\
\hline
22 &  +  &  -  &  +  &  * \\
\hline
23 &  +  &  +  &  +  &  + \\
\hline
24 &  *  &  *  &  *  &  * \\
\hline
25 &  +  &  +  &  +  &  + \\
\hline
26 &  +  &  +  &  +  &  + \\
\hline
27 &  *  &  +  &  +  &  + \\
\hline
28 &  +  &  +  &  +  &  + \\
\hline
29 &  +  &  +  &  +  &  + \\
\hline
30 &  *  &  *  &  -  &  - \\
\hline
31 &  *  &  *  &  -  &  * \\
\hline
32 &  +  &  +  &  +  &  * \\
\hline
33 &  +  &  +  &  +  &  * \\
\hline
\end{tabular} 
\end{table}
\egroup

Remarks are in order on some of the problems.
\begin{description}
	\item[Experiment 1] This is an example of the phenomenon described in the beginning of Section~\ref{directfixedpoint}: the closed-loop matrix $A-GX_s$ is not diagonalizable. The coefficient matrices for this example are 
\[A=\begin{bmatrix}
	0 & 1\\ 0 & 0
\end{bmatrix}
, G=\begin{bmatrix}
	0 & 0\\ 0 & 1
\end{bmatrix}
\quad \textrm{and} \quad
Q=\begin{bmatrix}
	1 & 0\\ 0 & 2
\end{bmatrix}.
\]
The exact value of the closed loop matrix for the original and transformed equations are respectively
\[
A-GX_s = \begin{bmatrix}
0 & 1\\ -1 & -2
\end{bmatrix}
\quad \text{and} \quad A_P-G_PY_s = 
\begin{bmatrix}
-2/3 & 1/3 \\ -1/3 & -4/3
\end{bmatrix},
\]
both with a double (defective) eigenvalue in $-1$. The approximation $\check{X}$ computed with the Schur method satisfies $\norm{X_s-\check{X}}=1.68e-15$. The matrix $A-G\check{X}$ is diagonalizable with two very close eigenvalues. Hence, the computed condition numbers of $V$ and $V_P$ are both large, and the first three algorithms, which are based on the diagonalization of an approximation of $A-GX_s$, fail. On the other hand, the fixed-point algorithm does not encounter any difficulty and returns a tight interval $\X$ containing the stabilizing solution. The condition number of the eigenvector matrix of $\mi(A-G\X)$ is $7.75e7$, but the verification with Algorithm~\ref{algo:stabilizing} succeeds nevertheless.
\item[Experiments 30 and 31] In Method~F for problem 30 and Method~K for problem 31, we report termination in a finite number of iterations, but NaN for the error. In these problems, the verification algorithm succeeds for the Riccati equation~\eqref{carep}, but the resulting interval $\Y$ cannot be converted into a solution interval $\X$ for~\eqref{care} using Lemma~\ref{changecare}, because the interval matrix $\mathbf{U}_1$ computed as described in Section~\ref{verifyingdifferent} contains singular matrices, hence the solution set $\X$ is unbounded. So the method fails to produce a solution enclosure for~\eqref{care}.
\end{description}

Another interesting observation is that Method~K needs only one iteration in all experiments when it works apart from one case (Experiment 6), i.e., the crucial relation \eqref{computationalexistence} is already fulfilled for $k=1$ in all the other cases.

When they are successful, Methods~H and~K are comparable with respect to execution time as well as with respect to the quality of the enclosure. However, there are cases in which Method~H is not successful, and this comprises cases with small dimensions (e.g. 2 in Example 10)
as well as cases with large dimensions (e.g. 397 in Example 27). 

Method~M is significantly faster than the other algorithms. We remark, though, that MATLAB, being an interpreted language, is often not reliable in evaluating computational times. In particular, INTLAB is implemented entirely in MATLAB code, and its running time does not always match the theoretical complexity, especially when dealing with small matrices. For Methods K and F, which rely on Algorithm~\ref{algo:CAREverification}, another consideration is that the computation of the matrix $P$ using the toolbox~\cite{pgdoubling2012} requires in its default implementation a tight double \texttt{for} loop on the matrix entries. MATLAB executes loops of this kind much more slowly than operations on full matrices; hence comparing running times may show a larger discrepancy than the actual difference in performance between the algorithms.

Methods~K and~M are the most reliable, and fail only on very ill-conditioned examples. Interestingly, the errors obtained by the two approaches differ by orders of magnitude on several problems, in both directions; there are also examples in which either one fails while the other succeeds. So there is no clear winner among the two.

Method~F has the largest number of failures. Despite that, it is useful in special cases (such as in Experiment~1) in which the other algorithms have difficulties, particularly when the closed-loop matrix is not diagonalizable.

In many of the examples the performance of the methods based on diagonalizing the closed-loop matrix is (loosely) related to the condition number of $V$ (or $V_P$, when it is used). To visualize this relationship, we show in Figures~\ref{fig:condH}--\ref{fig:condF} scatter plots of the obtained accuracy vs. the value of this condition number in the various examples. 
When the magnitude of $\operatorname{cond}(V)$ is moderate, $\operatorname{cond}(V_P)$ has typically the same order of magnitude, but in some cases when $\operatorname{cond}(V)$ is large $\operatorname{cond}(V_P)$ seems to be considerably lower, as shown in Figure~\ref{fig:condcond}. There is only one case in which $\operatorname{cond}(V_P)$ is considerably larger than $\operatorname{cond}(V)$, that is, Experiment 9 (1.22 vs. 69.8). This shows experimentally that switching from the formulation~\eqref{care} to~\eqref{carep} is beneficial.

\subsection{Experiments with varying sizes}
In view of the fact that the three Methods H, K and~F are iterative taking an unspecified number of steps, and that the last two require a factorization which may require $\mathcal{O}(n^3\log_\tau n)$ in the worst case, when $n$ is the size of $X$ in~\eqref{care}, the reader may wonder how the time taken by the various algorithm scales with the dimension $n$ in practice. We have tested all algorithms on~\cite[Problem~15]{carex1995}, which is a problem designed explicitly to check how Riccati solvers scale with the dimension of the equation. We have generated the test problem in 30 different sizes equally distributed in logarithmic scale between $10$ and $1000$, and we have tested the four algorithms on these examples. The resulting CPU times are reported in Figure~\ref{fig:scaled}.

\pgfplotstableread{results_scaled.table}{\resultsscaled}
\begin{figure}
\centering
\begin{tikzpicture}
\begin{loglogaxis}[width=\textwidth, xlabel=$n$, ylabel={CPU time}, legend pos = north west]
\addplot table[x=n,y=hashemi.time] \resultsscaled; \addlegendentry{Method H}
\addplot table[x=n,y=miyajima.time] \resultsscaled; \addlegendentry{Method M}
\addplot table[x=n,y=our.time] \resultsscaled; \addlegendentry{Method K}
\addplot table[x=n,y=fixed.time] \resultsscaled; \addlegendentry{Method F}
\addplot[domain=100:1000,dashed] {1e-5*x^3}; \addlegendentry{$f(x) = 3\cdot 10^{-5}x^3$}
\end{loglogaxis}
\end{tikzpicture}
\caption{CPU times for verification on a scaled version of Experiment 15 vs. dimension $n$.} \label{fig:scaled}
\end{figure}
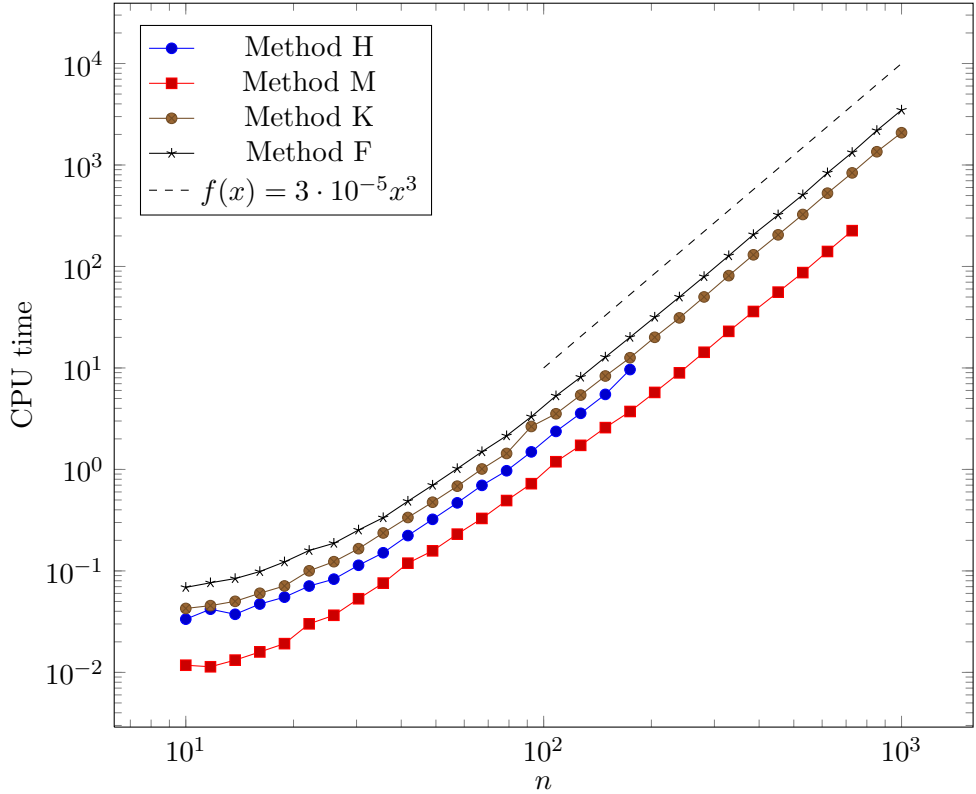

Overall, the results shows that all methods scale essentially with $\mathcal{O}(n^3)$, and in particular that Methods~K and~F stay within a moderate factor of the time taken by Method M. In the two largest experiments $n=853, n=1000$, Method K is the only one to succeed: Method M fails, while Method F delivers a solution enclosure for which the stabilizing property cannot be proved. Method H fails for each $n\geq 204$. Verification of the stabilizing property succeeds in all other cases apart from the two mentioned above for Method F.

The MATLAB code used for the experiments is available online on~\url{https://bitbucket.org/fph/verifiedriccati}.


\section{Summary and Outlook}\label{summary}

We have introduced several improvements to the method in~\cite{Hashemi2012}, borrowing ideas from both the interval arithmetic and the matrix equations literature. The resulting method has been tested on several standard benchmark experiments, and is competitive with the one introduced in~\cite{Miyajima2015}, returning a smaller solution enclosure in several of the experiments. Moreover, the new fixed-point method described in Section~\ref{directfixedpoint} is a useful addition to the battery of existing verification methods; it is especially useful in the cases in which the closed-loop matrix is not diagonalizable.

There is no single algorithm that beats all the others on all the benchmark problems; hence it is important to have several methods available, each with its strengths and drawbacks. Overall, all but two of the problems in this challenging set of experiments could be verified with success.

A number of open problems remain: first of all reducing to zero the number of remaining failures in the methods. Of particular interest would be a method more effective than Method~F that does not rely on the closed loop matrix being diagonalizable. Other possible research lines are applying these approaches to discrete-time Riccati equations (\emph{DARE}) or more generally to non-symmetric algebraic Riccati equations (\emph{NARE}). 


\section{Acknowledgments}\label{acknowledgment}
The authors thank Prof.~Dr.~Wolfram Luther for providing them the technical report related to the reference~\cite{LutherOttenTraczinski1998} and also the paper~\cite{LutherOtten1999}.

T.~Haqiri acknowledges the support by the Ministry of Science, Research and Technology of the Islamic Republic of Iran for her abroad research scholarship.

F.~Poloni acknowledges the support of INDAM (Istituto Nazionale di Alta Matematica) and of the PRA 2014 project ``Mathematical models for complex networks and systems'' of the university of Pisa.


\bibliographystyle{plainurl}
\bibliography{intervals}

\begin{thebibliography}{10}

\bibitem{IEEE1985}
{\em {IEEE standard for binary floating-point arithmetic}}.
\newblock Institute of Electrical and Electronics Engineers, New York, 1985.
\newblock Note: Standard 754--1985.

\bibitem{Alefeld1983}
G{\"o}tz Alefeld and J{\"u}rgen Herzberger.
\newblock {\em Introduction to interval computations}.
\newblock Computer Science and Applied Mathematics. Academic Press, Inc.
  [Harcourt Brace Jovanovich, Publishers], New York, 1983.
\newblock Translated from the German by Jon Rokne.

\bibitem{BartelsStewart1972}
R.~H. Bartels and G.~W. Stewart.
\newblock Solution of the matrix equation ax + xb = c [f4].
\newblock {\em Commun. ACM}, 15(9):820--826, September 1972.
\newblock URL: \url{http://doi.acm.org/10.1145/361573.361582}, \href
  {http://dx.doi.org/10.1145/361573.361582} {\path{doi:10.1145/361573.361582}}.

\bibitem{carex1995}
P.~Benner, A.~Laub, and V.~Mehrmann.
\newblock A collection of benchmark examples for the numerical solution of
  algebraic {R}iccati equations {I}: the continuous-time case.
\newblock Technical Report SPC 95-22, Forschergruppe `Scientific Parallel
  Computing', Fakult\"at f\"ur Mathematik, TU Chemnitz-Zwickau, 1995.
\newblock Version dated February 28, 1996.

\bibitem{BiniIannazzoMeini2012}
Dario~A. Bini, Bruno Iannazzo, and Beatrice Meini.
\newblock {\em Numerical solution of algebraic {R}iccati equations}, volume~9
  of {\em Fundamentals of Algorithms}.
\newblock Society for Industrial and Applied Mathematics (SIAM), Philadelphia,
  PA, 2012.

\bibitem{Brockett_book}
R.W. Brockett.
\newblock {\em Finite dimensional linear systems}.
\newblock Series in decision and control. Wiley, 1970.

\bibitem{ChuLiuMehrmann2007}
Delin Chu, Xinmin Liu, and Volker Mehrmann.
\newblock A numerical method for computing the {H}amiltonian {S}chur form.
\newblock {\em Numer. Math.}, 105(3):375--412, 2007.
\newblock URL: \url{http://dx.doi.org/10.1007/s00211-006-0043-0}, \href
  {http://dx.doi.org/10.1007/s00211-006-0043-0}
  {\path{doi:10.1007/s00211-006-0043-0}}.

\bibitem{Frommer2001}
Andreas Frommer.
\newblock Proving conjectures by use of interval arithmetic.
\newblock In {\em Perspectives on enclosure methods ({K}arlsruhe, 2000)}, pages
  1--13. Springer, Vienna, 2001.

\bibitem{Frommer2009}
Andreas Frommer and Behnam Hashemi.
\newblock Verified computation of square roots of a matrix.
\newblock {\em SIAM J. Matrix Anal. Appl.}, 31(3):1279--1302, 2009.
\newblock URL: \url{http://dx.doi.org/10.1137/090757058}, \href
  {http://dx.doi.org/10.1137/090757058} {\path{doi:10.1137/090757058}}.

\bibitem{Frommer2012}
Andreas Frommer and Behnam Hashemi.
\newblock Verified error bounds for solutions of {S}ylvester matrix equations.
\newblock {\em Linear Algebra Appl.}, 436(2):405--420, 2012.
\newblock URL: \url{http://dx.doi.org/10.1016/j.laa.2010.12.002}, \href
  {http://dx.doi.org/10.1016/j.laa.2010.12.002}
  {\path{doi:10.1016/j.laa.2010.12.002}}.

\bibitem{Frommer2014}
Andreas Frommer, Behnam Hashemi, and Thomas Sablik.
\newblock Computing enclosures for the inverse square root and the sign
  function of a matrix.
\newblock {\em Linear Algebra Appl.}, 456:199--213, 2014.
\newblock URL: \url{http://dx.doi.org/10.1016/j.laa.2013.11.047}, \href
  {http://dx.doi.org/10.1016/j.laa.2013.11.047}
  {\path{doi:10.1016/j.laa.2013.11.047}}.

\bibitem{GolubVanLoan2013}
Gene~H. Golub and Charles~F. Van~Loan.
\newblock {\em Matrix computations}.
\newblock Johns Hopkins Studies in the Mathematical Sciences. Johns Hopkins
  University Press, Baltimore, MD, fourth edition, 2013.

\bibitem{Hargreaves2002}
G.~I. Hargreaves and Gareth~I. Hargreaves.
\newblock Interval analysis in {MATLAB}.
\newblock Technical report, Department of Mathematics, University of
  Manchester, Manchester, 2002.

\bibitem{Hashemi2012}
Behnam Hashemi.
\newblock Verified computation of symmetric solutions to continuous-time
  algebraic {R}iccati matrix equations.
\newblock In {\em SCAN, 15'th GAMM-IMACS International Symposium on Scientific
  Computing, Computer Arithmetic and Verified Numerical Computations}, pages
  54--56. Russian Academy of Sciences, 2012.
\newblock With accompanying slides available online.
\newblock URL:
  \url{http://conf.nsc.ru/files/conferences/scan2012/139586/Hashemi-scan2012.pdf}.

\bibitem{HashemiDehghan2010}
Behnam Hashemi and Mehdi Dehghan.
\newblock Efficient computation of enclosures for the exact solvents of a
  quadratic matrix equation.
\newblock {\em Electron. J. Linear Algebra}, 20:519--536, 2010.

\bibitem{Higham2008}
Nicholas~J. Higham.
\newblock {\em Functions of matrices}.
\newblock Society for Industrial and Applied Mathematics (SIAM), Philadelphia,
  PA, 2008.
\newblock Theory and computation.
\newblock URL: \url{http://dx.doi.org/10.1137/1.9780898717778}, \href
  {http://dx.doi.org/10.1137/1.9780898717778}
  {\path{doi:10.1137/1.9780898717778}}.

\bibitem{Horn1994}
Roger~A. Horn and Charles~R. Johnson.
\newblock {\em Topics in matrix analysis}.
\newblock Cambridge University Press, Cambridge, 1994.
\newblock Corrected reprint of the 1991 original.

\bibitem{Kearfott2005}
R.~B. Kearfott, M.~T. Nakao, A.~Neumaier, S.~M. Rump, S.~P. Shary, and P.~V.
  Hentenryck.
\newblock Standardized notation in interval analysis.
\newblock In {\em Proc. XIII Baikal International School-seminar
  “Optimization methods and their applications}, volume~4, pages 106--113,
  2005.

\bibitem{Krawczyk1969}
R.~Krawczyk.
\newblock Newton-algorithms for evaluation of roots with error bounds.
\newblock {\em Computing}, 4(3):187--201.
\newblock URL: \url{http://dx.doi.org/10.1007/BF02234767}, \href
  {http://dx.doi.org/10.1007/BF02234767} {\path{doi:10.1007/BF02234767}}.

\bibitem{LancasterRodman1995}
Peter Lancaster and Leiba Rodman.
\newblock {\em Algebraic {R}iccati equations}.
\newblock Oxford Science Publications. The Clarendon Press, Oxford University
  Press, New York, 1995.

\bibitem{LutherOttenTraczinski1998}
W.~Luther, W.~Otten, and H.~Traczinski.
\newblock {\em Verified Calculation of Solutions of Continuous and Discrete
  Time Algebraic Riccati Equation}.
\newblock Number 422 in Schriftenreihe des Fachbereichs Mathematik.
  Universit{\"a}t Duisburg, Fachbereich Mathematik, 1998.

\bibitem{LutherOtten1999}
Wolfram Luther and Werner Otten.
\newblock Verified calculation of the solution of algebraic {R}iccati equation.
\newblock In {\em Developments in reliable computing (Budapest, 1998)}, pages
  105--118. Kluwer Acad. Publ., Dordrecht, 1999.

\bibitem{MehrmannPoloni2012}
Volker Mehrmann and Federico Poloni.
\newblock Doubling algorithms with permuted {L}agrangian graph bases.
\newblock {\em SIAM J. Matrix Anal. Appl.}, 33(3):780--805, 2012.
\newblock URL: \url{http://dx.doi.org/10.1137/110850773}, \href
  {http://dx.doi.org/10.1137/110850773} {\path{doi:10.1137/110850773}}.

\bibitem{Miy14}
Shinya Miyajima.
\newblock Fast enclosure for all eigenvalues and invariant subspaces in
  generalized eigenvalue problems.
\newblock {\em SIAM J. Matrix Anal. Appl.}, 35(3):1205--1225, 2014.
\newblock URL: \url{http://dx.doi.org/10.1137/140953150}, \href
  {http://dx.doi.org/10.1137/140953150} {\path{doi:10.1137/140953150}}.

\bibitem{Miyajima2015}
Shinya Miyajima.
\newblock Fast verified computation for solutions of continuous-time algebraic
  {R}iccati equations.
\newblock {\em Jpn. J. Ind. Appl. Math.}, 32(2):529--544, 2015.
\newblock URL: \url{http://dx.doi.org/10.1007/s13160-015-0178-4}, \href
  {http://dx.doi.org/10.1007/s13160-015-0178-4}
  {\path{doi:10.1007/s13160-015-0178-4}}.

\bibitem{Moore2009}
Ramon~E. Moore, R.~Baker Kearfott, and Michael~J. Cloud.
\newblock {\em Introduction to interval analysis}.
\newblock Society for Industrial and Applied Mathematics (SIAM), Philadelphia,
  PA, 2009.
\newblock URL: \url{http://dx.doi.org/10.1137/1.9780898717716}, \href
  {http://dx.doi.org/10.1137/1.9780898717716}
  {\path{doi:10.1137/1.9780898717716}}.

\bibitem{pgdoubling2012}
Federico Poloni.
\newblock {PGDoubling} -- a {MATLAB} package to solve algebraic {R}iccati
  equations and optimal control problems using permuted graph bases, 2012.
\newblock URL: \url{https://bitbucket.org/fph/pgdoubling}.

\bibitem{Rump2010}
Siegfried~M. Rump.
\newblock Verification methods: rigorous results using floating-point
  arithmetic.
\newblock {\em Acta Numer.}, 19:287--449, 2010.
\newblock URL: \url{http://dx.doi.org/10.1017/S096249291000005X}, \href
  {http://dx.doi.org/10.1017/S096249291000005X}
  {\path{doi:10.1017/S096249291000005X}}.

\bibitem{Intlab1999}
{S.M.} Rump.
\newblock {INTLAB - INTerval LABoratory}.
\newblock In Tibor Csendes, editor, {\em {Developments~in~Reliable Computing}},
  pages 77--104. Kluwer Academic Publishers, Dordrecht, 1999.
\newblock URL: \url{http://www.ti3.tuhh.de/rump/}.

\bibitem{Wachspress1988}
Eugene~L. Wachspress.
\newblock Iterative solution of the {L}yapunov matrix equation.
\newblock {\em Appl. Math. Lett.}, 1(1):87--90, 1988.
\newblock URL: \url{http://dx.doi.org/10.1016/0893-9659(88)90183-8}, \href
  {http://dx.doi.org/10.1016/0893-9659(88)90183-8}
  {\path{doi:10.1016/0893-9659(88)90183-8}}.

\bibitem{YanoKoga2011}
K.~{Yano} and M.~{Koga}.
\newblock Verified numerical computation in lq control problem.
\newblock {\em Transactions of the Society of Instrument and Control
  Engineers}, 45:261--267, 2011.

\end{thebibliography}
\end{document}